\documentclass[11pt]{amsart}

\usepackage[utf8]{inputenc}
\usepackage[T1]{fontenc}
\usepackage[margin=3.7cm]{geometry}        % balanced margins
\usepackage[foot]{amsaddr}
\usepackage{amssymb}
\usepackage{mathtools}
\usepackage{graphicx}
\usepackage[font=small,labelfont=bf,width=\linewidth]{caption}
\usepackage{xcolor}
\usepackage{esint}
\usepackage{subcaption}

\usepackage{algorithm}
\usepackage{algorithmic}

\usepackage[shortlabels]{enumitem}

\numberwithin{equation}{section}

\usepackage[sort&compress,numbers]{natbib}

\theoremstyle{plain}
\newtheorem{theorem}{Theorem}[section]
\newtheorem{proposition}{Proposition}[section]

\theoremstyle{definition}

\theoremstyle{remark}
\newtheorem{remark}{Remark}[section]

\usepackage[colorlinks=true,
            citecolor=blue,
            urlcolor=blue,
            linkcolor=blue]{hyperref}
\usepackage[nameinlink,capitalize]{cleveref}

\crefformat{equation}{(#2#1#3)}
\crefrangeformat{equation}{(#3#1#4) to~(#5#2#6)}
\crefmultiformat{equation}{(#2#1#3)}{ and~(#2#1#3)}{, (#2#1#3)}{ and~(#2#1#3)}

\newcommand{\R}{\mathbb{R}}
\newcommand{\N}{\mathbb{N}}
\renewcommand{\vec}[1]{{#1}}
\newcommand{\mat}[1]{{#1}}
\newcommand{\ip}[2]{\langle #1, #2 \rangle}
\newcommand{\ddt}{\frac{\rm d}{{\rm d} t}}
\newcommand{\abs}[1]{\left\vert #1 \right\vert}
\newcommand{\de}{\,\mathrm{d}}

\DeclareMathOperator{\Diag}{Diag}

\newcommand{\methodname}{{SILAS}}

\title[Data-driven discovery of bounded polynomial ODEs]{Data-driven discovery of polynomial ODEs with provably bounded solutions}
\author{Albert Alcalde, Giovanni Fantuzzi}
\address{Dept. Mathematics, Friedrich--Alexander Universit\"at Erlangen--N\"urnberg}
\email{albert.alcalde@fau.de}
\email{giovanni.fantuzzi@fau.de}
\date{\today}

\begin{document}
\begin{abstract}
    We introduce \methodname, a data-driven framework for discovering polynomial ordinary differential equations (ODEs) with provably bounded trajectories. Boundedness is certified by compact absorbing sets defined via polynomial Lyapunov functions. We jointly identify the ODE vector field and the Lyapunov function using a well-posed nonconvex optimization problem built using polynomial optimization tools. We solve this problem using an alternating block-coordinate optimization scheme with convex subproblems, whose feasibility is ensured by a novel model-agnostic initialization that identifies a candidate Lyapunov function from data. Our methods extend prior approaches for quadratic ODEs with absorbing ellipsoids to a significantly broader class of ODEs and absorbing sets. A suite of over 100 examples demonstrates that \methodname\ can recover accurate and provably bounded ODE models for a broad range of nonlinear dynamical systems.
\end{abstract}

% Needs to be after abstract in amsart
\maketitle

% Content written separately
\section{Introduction}\label{s:intro}
A central problem in data-driven modeling of dynamical systems is to learn models that not only reproduce observed data with high fidelity, but are also consistent with known structural properties of the true underlying dynamics. For instance, a model may be required to respect a known energy conservation principle \cite{greydanus2019hamiltonian,doi:10.1073/pnas.1814058116}, or be invariant under a specified symmetry transformation \cite{bronstein2021geometric,otto2025unified}. In this work, we focus on learning continuous-time models in the form of ordinary differential equations (ODEs) with provably bounded trajectories. In particular, we discover ODEs with bounded absorbing sets, that is, `trapping regions' in the state space within which all ODE solutions eventually remain. Such `bounded ODEs' typically arise when studying dissipative dynamics and are highly desirable in data-driven modeling problems because they are guaranteed to remain well-posed in regimes beyond their training data. 

\subsection{Contributions}
We develop a framework called \methodname\ (System Identification with Lyapunov-certified Absorbing Sets) for the data-driven discovery of provably bounded ODEs with polynomial vector fields of arbitrary degree. We accomplish this by combining a well-known characterization of absorbing sets via Lyapunov functions with sum-of-squares techniques from polynomial optimization, which allow one to construct polynomial ODEs consistent with a given Lyapunov function~\cite{ahmadi2023side} as well as polynomial Lyapunov functions for given polynomial ODEs~\cite{Goluskin2020}. Specifically, our contributions can be summarized as follows.
\begin{enumerate}[a), itemsep=0.15ex,
            leftmargin = \parindent, % item text indentation
            labelsep = *, % item label at margin
            widest = 9,
            ]
    \item We show that bounded polynomial ODEs with absorbing sets characterized by polynomial Lyapunov functions can be discovered from data through a well-posed nonconvex optimization problem with polynomial sum-of-squares constraints.
    
    \item We compute a feasible solution to this nonconvex optimization problem using a `block-coordinate' optimization algorithm, which alternately updates the ODE vector field and the Lyapunov function.
    Each update requires solving a convex problem over sum-of-squares polynomials, which can be posed as a semidefinite program and solved with established algorithms \cite{Lasserre2001,Parrilo2003,Nesterov2000,Laurent2009,Lasserre2010,Parrilo2013}.
    
    \item Crucially, we adapt ideas from \cite{Bramburger2024} to initialize our block-coordinate optimization algorithm with a Lyapunov function discovered from data in a model-agnostic way, rather than with an ODE identified without boundedness constraints. This ensures that every block-coordinate optimization subproblem is feasible and that our method always returns a bounded ODE.
    
    \item We demonstrate our \methodname\ framework on numerous nonlinear dynamics modeling problems, including ODE discovery problems for 109 chaotic systems in the \texttt{dysts} library~\cite{gilpin2021chaos} and a reduced-order modeling problem for a reaction-diffusion partial differential equation. These computations show that \methodname\ performs robustly if given sufficient data, also on non-polynomial systems.
\end{enumerate}

\subsection{Related work}
Absorbing sets and Lyapunov functions have already been used to discover bounded ODEs from data. For example, Lyapunov characterizations of absorbing sets have been employed in~\cite{Tang2024,goertzen2025learning} to learn neural ODEs and discrete-time neural-network maps with absorbing ellipsoids described by quadratic Lyapunov functions. The idea also underpins approaches to discover quadratic polynomial ODEs with so-called `lossless' (or `energy-preserving') quadratic nonlinearities and absorbing ellipsoids~\cite{SchlegelNoack2015,kaptanoglu2021, peng2025extending,liao2025lossless}. The term `lossless' here indicates that the homogeneous quadratic part of the Lyapunov function is a conserved quantity for the dynamics governed by the quadratic nonlinearity alone. For such systems, the ODE discovery problem requires optimizing the ODE and Lyapunov function parameters to best fit the data while making a certain matrix with bilinear dependence on the parameters positive semidefinite. 
\methodname\ generalizes this framework to arbitrary-degree polynomial ODEs and Lyapunov functions, reducing to it in the quadratic case, where the Lyapunov constraints enforce lossless nonlinearities (similar constraints apply to all even-degree polynomial ODEs, see \cref{rem:lossless-nnl}).

Lyapunov-based learning techniques also arise in imitation learning, where the goal is to discover ODE models for dynamical systems with globally asymptotically stable equilibria. Early approaches~\cite{5953529} impose global asymptotic stability through a fixed quadratic Lyapunov function (see also~\cite{ahmadi2023side} for a broader related framework for learning dynamics with `side constraints'), while more recent works~\cite{pmlr-v229-abyaneh23a,abyaneh2024globally} jointly learn the ODE and Lyapunov function from data. In particular, the techniques in~\cite{pmlr-v229-abyaneh23a} rely on polynomial parameterizations and sum-of-squares certificates of polynomial non-negativity similar to those used in this work, and may be viewed as a version of \methodname\ with a prescribed equilibrium point as the absorbing set. We also stress that there is a broad literature on applying sum-of-squares techniques to study nonlinear dynamics governed by known polynomial ODEs; see, for instance, \cite{Antonis2005,Lasserre2008,Parrilo2013,Henrion2014,Korda2014,Fantuzzi2016,Goluskin2020} for a non-exhaustive selection of works in this area.

From a computational perspective, the challenges faced by \methodname\ are similar to those stemming from the bilinear matrix inequality associated with quadratic models with absorbing ellipsoids. The `trapping SINDy' algorithm~\cite{kaptanoglu2021} handles this bilinear matrix inequality with a penalty method, while~\cite{GOYAL2025134893} enforces it exactly through an explicit nonlinear parameterization of lossless quadratic ODEs inspired by the structure of linear port-Hamiltonian systems.
The works~\cite{liao2025lossless,heide2025}, instead, bypass the bilinearity by alternately updating the ODE and Lyapunov function parameters after an unconstrained ODE discovery step. Our block-coordinate optimization strategy does the same for polynomial ODEs and Lyapunov functions of arbitrary degree, but with one crucial difference: we start from a Lyapunov function discovery step to ensure the feasibility of all subsequent updates. In the quadratic case, therefore, we obtain a provably feasible version of the methods of~\cite{heide2025}.

Finally, the idea of initializing bilinear optimization solvers using Lyapunov functions learned from data has already been proposed in the context of estimating regions of attraction of equilibria for known ODEs~\cite{Topcu2008}. The approach we take within \methodname, based on ideas in~\cite{Bramburger2024}, is different because we only have access to data, not to a known ODE. As such, it fits within the broader literature on approximating Lyapunov functions and absorbing sets from data. We mention in particular the works~\cite{Korda2020,Tacchi2025}, which use data and convex optimization to approximate Lyapunov functions describing regions of attraction, 
and the works~\cite{mauroy2023,Mauroy2020}, which approximate regions of attraction and absorbing sets with data-driven estimates of eigenfunctions of the Koopman operator associated to the system's dynamics. As we shall explain, our \methodname\ framework also combines convex optimization and data-driven approximations of the Koopman operator (precisely, of its generator), but differs from the aforementioned approaches in the Koopman operator literature as we do not use Koopman eigenfunctions.

\subsection{Outline}

Our methods are presented in \cref{s:methods}. 
Specifically, \cref{ss:lyapunov-functions,ss:model-based-lyap} review how to describe absorbing sets with Lyapunov functions and how to construct polynomial Lyapunov functions for known polynomial ODEs. In \cref{ss:data-driven-lyap}, we describe a model-agnostic way to discover absorbing sets from data. We combine these techniques into our \methodname\ framework in \cref{ss:overall-method}. The broad applicability of \methodname\ is demonstrated in \cref{s:numerics} through an extensive suite of examples. Conclusions and perspectives for further work are offered in \cref{s:conclusion}.
\section{Methods}\label{s:methods}

Consider a continuous-time, autonomous dynamical system with state space $\R^n$. We collect $m$ approximate measurements of the system's state $\vec{x}$ and of its rate-of-change $\dot{\vec{x}}=\ddt\vec{x}$ into a dataset
\begin{equation}\label{e:dataset}
    \mathcal{D} = \left\{ (\vec{x}_1,\vec{y}_1), \, \ldots,\, (\vec{x}_m, \vec{y}_m)\right\},
\end{equation}
where $\vec{x}_i \approx \vec{x}(t_i)$ and $\vec{y}_i \approx \dot{\vec{x}}(t_i)$ are the measurements at time $t_i$. The measurements need not be sequential and $\dot{x}(t_i)$ may be estimated from sequential state measurements using finite differences or any other derivative approximation.

Our goal is to identify a bounded polynomial ODE to model this data. Precisely, we look for a polynomial vector field $f:\R^n \to \R^n$  of degree $d_f \geq 1$ such that $f(\vec{x}_i)\approx \vec{y}_i$ for all data pairs $(\vec{x}_i,\vec{y}_i)$ and such that the solution $\vec{x}(t)$ of the ODE
\begin{equation}\label{e:ode}
    \dot{\vec{x}}(t) = \vec{f}(\vec{x}(t)),\quad \vec{x}(0)=\vec{x}_0
\end{equation}
remains bounded at all times and for every initial condition $\vec{x}_0 \in \R^n$. We will solve this model learning problem in \cref{ss:overall-method}, but first we introduce the tools needed to enforce the boundedness constraint.

\subsection{Proving boundedness with Lyapunov functions}\label{ss:lyapunov-functions}

Fix a vector field $f$, not necessarily polynomial. A classical approach to proving that the ODE~\cref{e:ode} has bounded solutions is to construct a continuously differentiable function $v:\R^n \to \R$, called a \emph{Lyapunov function}, satisfying suitable constraints. We will use two types of Lyapunov functions, distinguished by the constraints they satisfy. The first ones are given in the next theorem, where $v$ is said to be \emph{coercive} if $\|\vec{x}\|\to \infty$ implies $v(\vec{x})\to +\infty$.
\begin{theorem}\label{th:lyap-fun-0}
    If $v \in C^1(\R^n)$ is coercive and $-\ip{\vec{f}(\vec{x})}{\nabla v(\vec{x})} \geq 0$ for all $\vec{x}\in\R^n$, then every solution of the ODE  \cref{e:ode} is uniformly bounded in time.
\end{theorem}
\begin{proof}
    If $\vec{x}(t)$ solves \cref{e:ode}, then $\ddt v(\vec{x}(t)) = \ip{\vec{f}(\vec{x}(t))}{\nabla v(\vec{x}(t))} \leq 0$ for all $t$. Then, $\vec{x}(t) \in \{ \vec{x}\in\R^n:\; v(\vec{x})\leq v(\vec{x}_0)\}$ and this set is compact since $v$ is coercive.
\end{proof}
The second type of Lyapunov functions we use provide a sharper characterization of boundedness because they ensure the existence of a compact absorbing set. This is a compact set $\mathcal{A} \subset \R^n$ such that
\begin{equation}\label{e:absorbing-set}
    \lim_{t \to \infty} \, \min_{\vec{y} \in \mathcal{A}} \|\vec{x}(t) - y\|= 0
\end{equation}
for every initial condition $\vec{x}_0$. Sufficient conditions on Lyapunov functions to characterize compact absorbing sets are provided by the next theorem, whose proof is a well-known argument based on Gr\"onwall's inequality and is included for completeness.
\begin{theorem}\label{th:lyap-fun}
If $v \in C^1(\R^n)$ is coercive and there exist real numbers $\alpha>0$ and $b \geq 0$ such that
\begin{equation}\label{e:lyapunov-inequality}
    b - v(\vec{x}) - \alpha\,  \ip{\vec{f}(\vec{x})}{\nabla v(\vec{x})}  \geq 0 \quad \forall \vec{x} \in \R^n,
\end{equation}
then the set $\mathcal{A}=\{\vec{x}\in\R^n:\;v(\vec{x})\leq b\}$ is a compact absorbing set for ODE \cref{e:ode}.
\end{theorem}
\begin{proof}
The set $\mathcal{A}$ is closed by construction and bounded because $v$ is coercive, so it is compact. To prove it is absorbing, let $\vec{x}(t)$ solve \cref{e:ode}. Since $v$ is differentiable, the chain rule yields $\ddt v(\vec{x}(t)) = \ip{\vec{f}(\vec{x}(t))}{\nabla v(\vec{x}(t))}$ for all times $t$. Then, inequality \cref{e:lyapunov-inequality} implies that
$\ddt v(\vec{x}(t)) + \frac{1}{\alpha} v(\vec{x}(t)) \leq \frac{b}{\alpha}$. We can then apply Gr\"onwall inequality (see, e.g., \cite[Lemma~2.1]{DoeringGibbon1995}) to obtain
\begin{equation}\label{e:gronwall}
    v(\vec{x}(t)) \leq  b + \left[ v(\vec{x}_0) - b \right] {\rm e}^{-t/\alpha}
\end{equation} 
for all times. Now, if $v(\vec{x}_0) \leq b$ we conclude from this inequality that $v(\vec{x}(t)) \leq b$ for all times, so $\vec{x}(t) \in \mathcal{A}$ and \cref{e:absorbing-set} holds trivially. If $v(\vec{x}_0) > b$, instead, we conclude from~\cref{e:gronwall} that $v(\vec{x}(t)) \leq v(\vec{x}_0)$ and that $\limsup_{t\to\infty} v(\vec{x}(t)) \leq b$. We claim that these two inequalities, combined with the continuity of $v$, imply \cref{e:absorbing-set}. Indeed, suppose for the sake of contradiction that $\min_{y \in \mathcal{A}}\|\vec{x}(t) - y\| \geq \varepsilon > 0$ for all $t$. Then the trajectory $\vec{x}(t)$ remains in the compact set $\mathcal{B}=\{\vec{x}\in \R^n :\; \min_{y \in \mathcal{A}}\|\vec{x}(t) - y\| \geq \varepsilon,\, v(\vec{x})\leq v(\vec{x}_0)\}$ for all $t$. But then, $\limsup_{t\to\infty} v(\vec{x}(t)) \geq\liminf_{t\to\infty} v(\vec{x}(t)) \geq \min_{\vec{x}\in\mathcal{B}} v(\vec{x}) > b$, where the last inequality is strict because the set $\mathcal{A} \cap \mathcal{B}$ is empty by construction and because the continuous function $v$ attains its minimum on the compact set $\mathcal{B}$. This, however, contradicts the inequality  $\limsup_{t\to\infty} v(\vec{x}(t)) \leq b$ deduced from~\cref{e:gronwall}.
\end{proof} 

\begin{remark}
Despite sharing the same name, the Lyapunov functions in \cref{th:lyap-fun-0,th:lyap-fun} are fundamentally different from those used to study the stability of equilibria. The latter are not considered in this work. Since no confusion can arise, we slightly abuse common terminology and call \cref{e:lyapunov-inequality} the \emph{Lyapunov inequality}. Observe also that there is no loss of generality in assuming Lyapunov functions to be nonnegative. Indeed, any Lyapunov function $v$ must attain a finite minimum value $v_{\rm min}$ because it is coercive. Then, it is immediate to verify that $w=v-v_{\rm min}$ is a nonnegative Lyapunov function characterizing the same absorbing set as $v$.
\end{remark}

\subsection{Optimizing polynomial Lyapunov functions for polynomial ODEs}\label{ss:model-based-lyap}

While Lyapunov functions are usually difficult to construct, polynomial Lyapunov functions for polynomial ODEs can be searched for using convex optimization. We explain this next, focusing on the Lyapunov functions from \cref{th:lyap-fun}. The Lyapunov functions from \cref{th:lyap-fun-0} can be handled analogously.

Fix $\alpha>0$, an invertible matrix $\mat{\Lambda} \in \R^{n\times n}$, and a vector $\vec{\mu} \in \R^n$. Let $\R[\vec{x}]_{d}$ denote the space of $n$-variate polynomials of $\vec{x}=(x_1,\ldots,x_n)$ with real coefficients and degree at most $d$. Any polynomial Lyapunov function $v$ must have even degree $d_v\geq 2$, otherwise it cannot be coercive. Then, given a polynomial vector field $f$ of degree $d_f \geq 1$, we fix an even degree $d_v \geq 2$ and search for a Lyapunov function $v$ by solving the optimization problem
\begin{equation}\label{e:lyap-opt-1}
    \min_{\substack{v \in \R[\vec{x}]_{d_v} \\ b \in \R}} \; |b| \quad
    \text{s.t.} \quad 
    \begin{cases}
        v(\vec{x}) - \|\mat{\Lambda}\vec{x} + \vec{\mu}\|^2 \geq 0 &\forall \vec{x} \in \R^n, \\
        b - v(\vec{x}) - \alpha\,  \ip{\vec{f}(\vec{x})}{\nabla v(\vec{x})}  \geq 0 &\forall \vec{x} \in \R^n.
    \end{cases}
\end{equation}
The intuition behind this problem is simple: the first constraint ensures that $v$ is coercive, the second is the Lyapunov inequality \cref{e:lyapunov-inequality}, and minimizing $|b|$ is a proxy for minimizing the size of the absorbing set defined by $v$. Both constraints are nonnegativity constraints on polynomials that depend affinely on the optimization variables, which are the constant $b$ and the coefficients of $v$ with respect to any chosen basis for $\R[\vec{x}]_{d_v}$.
We could actually fix $\mat{\Lambda}=I$ and $\vec{\mu}=0$ in \cref{e:lyap-opt-1}, but we do not because a more judicious choice can considerably improve numerical conditioning.

Problem \cref{e:lyap-opt-1} is convex, but is computationally intractable because nonnegativity constraints on polynomials are NP-hard in general \cite{Murty1987}. Consequently, we follow an established approach from polynomial optimization and replace the inequality constraints with the stronger condition that their left-hand sides are sums of squares (SOS) of other polynomials. Specifically, writing
\begin{equation*}
    \Sigma[\vec{x}]_{d} := \left\{ \sigma \in \R[\vec{x}]_d:\;\exists k\in\N,\; \rho_1,\ldots,\rho_k \in \R[\vec{x}]_{\lfloor d/2\rfloor} \; \text{s.t.}\; \sigma = \rho_1^2 + \cdots + \rho_k^2  \right\}
\end{equation*}
for the set of SOS polynomials of degree no larger than $d$, we replace problem \cref{e:lyap-opt-1} with its strengthening
\begin{equation}\label{e:lyap-opt-sos}
    \min_{\substack{v \in \R[\vec{x}]_{d_v} \\ b \in \R}} \; |b| \quad
    \text{s.t.} \quad 
    \begin{cases}
        v(\vec{x}) -  \|\mat{\Lambda} \vec{x} + \vec{\mu}\|^2 \in \Sigma[\vec{x}]_{d_v}, \\
        b - v(\vec{x}) - \alpha\,  \ip{\vec{f}(\vec{x})}{\nabla v(\vec{x})}  \in \Sigma[\vec{x}]_{d_v + d_f - 1}.
    \end{cases}
\end{equation}
While not all nonnegative polynomials are SOS in general~\cite{Hilbert1888}, this strengthening is convenient because it can be cast as a semidefinite program \cite{Lasserre2001, Parrilo2003, Nesterov2000, Lasserre2010, Parrilo2013, Laurent2009}, so it can be solved with established interior-point algorithms \cite{Nesterov1994, Ye1997, Nemirovski2007}. 
\begin{remark}[Lossless nonlinearities]\label{rem:lossless-nnl}
    Since the Lyapunov function $v$ in problem \cref{e:lyap-opt-1} is coercive, the Lyapunov inequality cannot be satisfied unless the polynomial $\ip{\vec{f}}{\nabla v}$ has even degree. For this to hold with even-degree $f$, the leading homogeneous parts of $f$ and $v$, denoted $h_f$ and $h_v$, must satisfy $\ip{h_f}{\nabla h_v} = 0$. This condition generalizes the `lossless nonlinearity' assumption for quadratic $f$ and $v$ in~\cite{kaptanoglu2021,liao2025lossless,heide2025} and is accounted for in problem \cref{e:lyap-opt-sos} through the SOS condition $b - v(\vec{x}) - \alpha\,  \ip{\vec{f}(\vec{x})}{\nabla v(\vec{x})}  \in \Sigma[\vec{x}]_{d_v + d_f - 1}$. This is because polynomials in the set $\Sigma[\vec{x}]_{d_v + d_f - 1}$ have even degree $2\lfloor (d_v + d_f - 1) / 2 \rfloor$, so odd-degree leading terms in $\ip{\vec{f}}{\nabla v}$ are set to zero in the semidefinite program formulation of~\cref{e:lyap-opt-sos}.
\end{remark}

\subsection{Discovering Lyapunov functions from data}
\label{ss:data-driven-lyap}

The SOS approach to constructing Lyapunov functions described above requires an explicit polynomial ODE, which is not available in data-driven ODE discovery problems. We therefore adapt it to discover a candidate Lyapunov function only from the dataset $\mathcal{D}$ in \cref{e:dataset}.

The idea is to use the data $\vec{y}_i \approx \vec{f}(\vec{x}_i)$ to construct a polynomial approximation for the unknown term $\ip{\vec{f}}{\nabla v}$ in the Lyapunov inequality. This requires estimating the action of the operator $v \mapsto \ip{\vec{f}}{\nabla v}$, which is the generator of the Koopman semigroup associated to the ODE \cref{e:ode}, on a basis for the polynomial space $\R[\vec{x}]_{d_v}$. We accomplish this using the following version of \emph{generator extended dynamic mode decomposition}~\cite{williams2015data, klus2020data}.

Fix a polynomial vector $\varphi=(\varphi_1,\ldots,\varphi_p)$ whose $p=\smash{\binom{n+d_v}{n}}$ entries are a basis for the polynomial space $\R[\vec{x}]_{d_v}$. Let $\vec{u} \in \R^p$ be a coefficient vector and expand
\begin{equation*}
    v(x) = \ip{\vec{u}}{\varphi(\vec{x})}
    \quad\text{and}\quad
    \ip{\vec{f}}{\nabla v} = \sum_{j=1}^p u_j \ip{\vec{f}}{\nabla \varphi_j}.
\end{equation*}
Since we want $v$ to be a Lyapunov function for a polynomial ODE of degree $d_f$, we approximate each of the terms $\ip{\vec{f}}{\nabla \varphi_j}$ with a polynomial of degree $d_v+d_f-1$. Specifically, we fix a basis $\vartheta=(\vartheta_1,\ldots,\vartheta_q)$ for the polynomial space $\R[\vec{x}]_{d_f + d_v - 1}$, which has dimension $q=\smash{\binom{n+d_v+d_f-1}{n}}$, and approximate $\ip{\vec{f}}{\nabla \varphi_j} \approx \sum_{k=1}^q G_{jk} \vartheta_k$ where the coefficients $G_{jk}$ are entries of a matrix $\mat{G} \in \R^{p \times  q}$ that solves the least-squares minimization problem
\begin{equation}\label{e:lsq}
    \min_{\mat{G} \in \R^{p \times q}} \sum_{i=1}^{m} \sum_{j=1}^p  w_i \abs{ \ip{\vec{y}_i}{\nabla \varphi_j(\vec{x}_i)}  - \sum_{k=1}^q G_{jk} \ \vartheta_k(\vec{x}_i)}^2.
\end{equation}
The positive weights $w_1,\ldots, w_m$ in this problem determine the importance of the measurements in the dataset $\mathcal{D}$. A standard choice is $w_i=\frac1m$, but other choices may lead to better results~\cite{BramburgerColebrook2025}.
The optimal solution of \cref{e:lsq} is
\begin{equation*}
    \mat{G} = \left( \mat{A}\mat{W}\mat{B}^\top \right) \left( \mat{B}\mat{W}\mat{B}^\top\right)^{\dagger},
\end{equation*}
where $W = \Diag(w_1,\ldots,w_m)$ is the diagonal matrix of weights, $\dagger$ indicates the Moore--Penrose pseudoinverse, and the matrices $A\in \R^{p\times m}$ and $B\in\R^{q\times m}$ are given by
\begin{gather*}
    \mat{A} = \begin{bmatrix}
        \ip{\vec{y}_1}{\nabla \varphi_1(\vec{x}_1)} & \cdots & \ip{\vec{y}_m}{\nabla \varphi_1(\vec{x}_m)}\\
        \vdots & \ddots & \vdots \\
        \ip{\vec{y}_1}{\nabla \varphi_p(\vec{x}_1)} & \cdots & \ip{\vec{y}_m}{\nabla \varphi_p(\vec{x}_m)}
    \end{bmatrix}
    \quad\text{and}\quad
    \mat{B} =
    \begin{bmatrix}
        \vartheta_1(\vec{x}_1) & \cdots & \vartheta_1(\vec{x}_m)\\
        \vdots & \ddots & \vdots \\
        \vartheta_q(\vec{x}_1) & \cdots & \vartheta_q(\vec{x}_m)
    \end{bmatrix}.
\end{gather*}
Given the matrix $\mat{G}$ and any polynomial $v=\ip{\vec{u}}{\varphi} \in \R[x]_{d_v}$ with coefficients $\vec{u}\in\R^p$, we can then approximate
\begin{equation*}
    \ip{\vec{f}}{\nabla v} \approx \ip{\vec{u}}{\mat{G} \vartheta}.
\end{equation*}
In general, this is only an approximation because the dataset $\mathcal{D}$ used to construct the least-squares problem~\cref{e:lsq} consists of approximate measurements of a dynamical system's state and rate-of-change, and because the true dynamics need not be governed exactly by a polynomial ODE of the chosen degree $d_f$. However, motivated by the Weierstrass approximation theorem and convergence results for extended dynamic mode decomposition \cite{williams2015data,Bramburger2024,klus2020data,KordaMezic2018,Klus2016}, the approximation should be accurate on compact sets if $d_f$ is sufficiently large and the dataset is sufficiently rich in those sets.

Now that we have a data-driven estimate of $\ip{\vec{f}}{\nabla v}$ for all polynomials $v\in\R[\vec{x}]_{d_v}$, we can attempt to discover an absorbing set for the unknown dynamical system generating the data. We do so by solving the following approximation of \cref{e:lyap-opt-sos}, in which the Lyapunov inequality is replaced by its data-driven approximation
\begin{equation}\label{e:lyap-opt-data-cnstr}
    \min_{\substack{\vec{u} \in \R^p \\ b \in \R}} \; |b| \quad
    \text{s.t.} \quad 
    \begin{cases}
        \ip{\vec{u}}{\varphi(\vec{x})} - \|\mat{\Lambda} \vec{x} + \vec{\mu}\|^2 \in \Sigma[\vec{x}]_{d_v}, \\
        b - \ip{\vec{u}}{\varphi(\vec{x}) + \alpha G \vartheta(\vec{x})}  \in \Sigma[\vec{x}]_{d_v + d_f - 1}.
    \end{cases}
\end{equation}
As before, we require $d_v \geq 2$ to be an even integer. If this problem is infeasible, we fix a penalty parameter $\kappa > 0$, select a computationally tractable norm for the polynomial space $\smash{\R[\vec{x}]_{d_v + d_f - 1}}$ (e.g., the $\ell^2$ norm of the vector of polynomial coefficients with respect to a convenient basis), and solve the penalized problem
\begin{equation}\label{e:lyap-opt-data-penalized}
    \begin{aligned}
        \min_{\substack{\vec{u} \in \R^p,\; b \in \R \\ \sigma \in \R[\vec{x}]_{d_v + d_f - 1}}} \quad 
        &|b| + \kappa \left\| b - \ip{\vec{u}}{\varphi + \alpha G \vartheta}  - \sigma \right\|_{\R[\vec{x}]_{d_v + d_f - 1}}
        \\
        \text{s.t.} \quad 
        &\ip{\vec{u}}{\varphi(\vec{x})} - \|\mat{\Lambda} \vec{x} + \vec{\mu}\|^2 \in \Sigma[\vec{x}]_{d_v}, \\
        &\sigma(\vec{x}) \in \Sigma[\vec{x}]_{d_v + d_f - 1}.
    \end{aligned}
\end{equation}
This problem is feasible because $\varphi(\vec{x})$ is a basis for $\R[\vec{x}]_{d_v}$ and $d_v \geq 2$, so it is always possible to choose $\vec{u}$ such that $\ip{\vec{u}}{\varphi(\vec{x})} - \|\mat{\Lambda} \vec{x} + \vec{\mu}\|^2=0\in\Sigma[x]_{d_v}$.

\subsection{The \methodname\ framework for bounded ODE discovery}\label{ss:overall-method}

We now combine the tools described above into our \methodname\ framework for discovering provably bounded polynomial ODEs from data.

Choose a model degree $d_f \geq 1$ and a Lyapunov function degree $d_v \geq 2$. Set $p = \smash{\binom{n+d_v}{n}}$, $r = \smash{\binom{n+d_f}{n}}$, and fix polynomial vectors
\begin{align*}
    \varphi(x) = \big(\varphi_1(x), \ldots, \varphi_p(x)\big) 
    \qquad\text{and}\qquad
    \zeta(x) = \big(\zeta_1(x), \ldots, \zeta_r(x)\big)
\end{align*}
whose entries are bases for the polynomial spaces $\R[x]_{d_v}$ and $\R[x]_{d_f}$, respectively. We parametrize the vector field of the polynomial ODE $\dot{\vec{x}}=\vec{f}(\vec{x})$ and the Lyapunov function certifying its boundedness as
\begin{equation*}
    f(\vec{x})= \mat{F} \zeta(\vec{x})
    \qquad\text{and}\qquad
    v(\vec{x}) = \ip{\vec{u}}{\varphi(\vec{x})},
\end{equation*}
respectively, where $\mat{F} \in \R^{n \times r}$ and $\vec{u}\in\R^p$ contain the ODE and Lyapunov function parameters. Note that $\nabla v = \nabla\varphi^\top u$  where the Jacobian matrix $\nabla \varphi : \R^n \to \R^{p \times n}$ has entries $(\nabla\varphi)_{ij} = \partial \varphi_i / \partial x_j$. 
Finally, we fix the following hyperparameters: an invertible matrix $\mat{\Lambda} \in \R^{n\times n}$, a vector $\vec{\mu} \in \R^n$, and positive real numbers $\alpha$, $\beta$, $\kappa$, $\varepsilon_1$, $\varepsilon_2$, $\varepsilon_3$ and $\varepsilon_4$.

Given the dataset $\mathcal{D}$ in \cref{e:dataset}, we discover a bounded polynomial ODE and a Lyapunov function certifying boundedness by solving the following optimization problem:
\begin{align}\label{eq:monster}
        \min_{\substack{
        \mat{F} \in \R^{n \times r}\\
        \vec{u} \in \R^p\\
        b,c \in \R
        }}
        \quad
        & \left(\sum_{i=1}^m \left\| \mat{F} \zeta(x_i) - y_i \right\|^2 \right)^{\frac12} + \varepsilon_1 |c-1| + \varepsilon_3 |b| + \varepsilon_2 \| \mat{F} \|_{\ell^1} 
        +\varepsilon_4 \|u\|_{\ell^1}
        \\ \nonumber
        \text{s.t.} \quad
        &b - c \ip{\vec{u}}{\varphi(\vec{x})} - \alpha \ip{\mat{F} \zeta(\vec{x})}{\nabla \varphi(\vec{x})^\top \vec{u}} \in \Sigma[\vec{x}]_{d_v+d_f-1},
        \\ \nonumber
        &\ip{\vec{u}}{\varphi(\vec{x})}  -  \|\mat{\Lambda} \vec{x} + \vec{\mu}\|^2 \in \Sigma[\vec{x}]_{d_v},
        \\ \nonumber
        &0\leq c \leq 1,
        \\ \nonumber 
        &|b| \leq \beta c.
\end{align}
This problem is not convex because the term $\ip{\mat{F} \zeta(\vec{x})}{\nabla \varphi(\vec{x})^\top \vec{u}}$ depends bilinearly on the optimization variables $\mat{F}$ and $\vec{u}$. However, as we show next, it is well-posed and any feasible point corresponds to an ODE with bounded solutions.
\begin{theorem}\label{th:monster-problem}
    The following statements hold for every choice of $\Lambda\in\R^{n\times n}$ invertible, $\vec{\mu}\in\R^n$, and $\alpha,\beta,\varepsilon_1,\varepsilon_2,\varepsilon_3,\varepsilon_4>0$:
    \begin{enumerate}[a), noitemsep]
        \item Problem \cref{eq:monster} is feasible and has an optimal solution.
        \item If $\mat{F}$ is admissible in \cref{eq:monster}, then the ODE $\dot{x}=\mat{F} \zeta(x)$ has bounded solutions.
    \end{enumerate}
\end{theorem}
\begin{proof}
    Problem \cref{eq:monster} is feasible because its constraints are satisfied when $b=0$, $c=0$, $\mat{F}=0$, and $\vec{u}$ is a vector such that $\ip{\vec{u}}{\varphi(\vec{x})}  =  \|\mat{\Lambda} \vec{x} + \vec{\mu}\|^2$, which exists since $\varphi(\vec{x})$ is a basis for $\R[x]_{d_v}$ and $d_v\geq 2$. 
    To see that the minimum is attained, note that the objective function is coercive, so every minimizing sequence is uniformly bounded and has a convergent subsequence. The limit of such a subsequence is an optimal solution for \cref{eq:monster} because the objective and constraints of this problem depend continuously on the optimization variables and because the set of SOS polynomials is closed (see, e.g., \cite[Corollary~3.5]{Laurent2009}).
    Finally, suppose $(\vec{u},\mat{F},b,c)$ is an admissible tuple. If $c=0$, then $b=0$ and solutions to the ODE $\dot{x}=\mat{F} \zeta(x)$ are bounded by \cref{th:lyap-fun-0}. If $c>0$, instead, boundedness follows from \cref{th:lyap-fun}.
\end{proof}
\begin{algorithm}[t]
\caption{The \methodname\ framework for bounded ODE discovery}
\label{alg:learn-ode}
\begin{algorithmic}
\REQUIRE Dataset $\mathcal{D}$\\
\REQUIRE ODE degree $d_f\geq 1$ and polynomial basis $\zeta$ for $\R[x]_{d_f}$
\REQUIRE Lyapunov function degree $d_v \geq 2$ and polynomial basis $\varphi$ for $\R[x]_{d_v}$
\REQUIRE A polynomial basis $\vartheta$ for $\R[x]_{d_v+d_f-1}$
\REQUIRE Parameters: $\Lambda \in \R^{n \times n}$ invertible, $\vec{\mu} \in \R^n$, $\alpha,\beta,\kappa,\varepsilon_1,\varepsilon_2,\varepsilon_3,\varepsilon_4 > 0$
\REQUIRE Maximum number of refinement iterations $K$
\ENSURE Bounded polynomial ODE and associated Lyapunov function

% \COMMENT{Step 1: Discover initial Lyapunov function}
\COMMENT{The SILAS algorithm}
\STATE $\vec{u}_0 \gets$ solve problem \cref{e:lyap-opt-data-cnstr} 
\IF{problem \cref{e:lyap-opt-data-cnstr} is infeasible}
  \STATE $\vec{u}_0 \gets$ solve problem \cref{e:lyap-opt-data-penalized}
\ENDIF
% \COMMENT{Step 2: Discover a bounded polynomial ODE}
\FOR{$k = 1$ to $K$}
  \STATE $(\mat{F}_k, b_{k-1/2}, c_k) \gets$ solve problem \cref{eq:monster} for fixed $\vec{u}=\vec{u}_{k-1}$
  \STATE $(\vec{u}_k, b_k) \gets$ solve problem \cref{eq:monster} for fixed $\mat{F}=\mat{F}_{k}$ and $c=c_k$
  \STATE \textbf{if} converged \textbf{then} \textbf{break}
\ENDFOR
\RETURN Bounded ODE $\dot{\vec{x}} = \mat{F}_k \zeta(\vec{x})$ and Lyapunov function $v(\vec{x}) = \ip{\vec{u}_k}{\varphi(\vec{x})}$
\end{algorithmic}
\end{algorithm}
Solving \cref{eq:monster} is difficult in practice because this problem is not convex. While recent techniques for nonconvex optimization over SOS polynomials can compute local minimizers \cite{Cunis2025,Olucak2025}, we follow \cite{heide2025} and settle for a simpler heuristic strategy summarized in \cref{alg:learn-ode}, which alternates between optimizing the triple $(\mat{F},b,c)$ for fixed $\vec{u}$ and optimizing the pair $(\vec{u},b)$ for fixed $(\mat{F},c)$. These are convex and semidefinite-representable problems, which can be solved using mature software~\cite{mosek,Lofberg2004}. We initialize the alternating minimization process using an optimal solution $(\vec{u}_0, b_0)$ of the data-driven Lyapunov function discovery problem \cref{e:lyap-opt-data-cnstr} or, if this is infeasible, of its penalized version \cref{e:lyap-opt-data-penalized}. We then iterate for a maximum of $K$ iterations or until the iterates stabilize (this, however, cannot be guaranteed without further assumptions and analysis, which we leave for future work; see, e.g. \cite{Bertsekas1999,Grippo1999,Beck2013} and references therein for more on the convergence of block-coordinate optimization methods). By \cref{th:monster-problem}, the matrix $\mat{F}$ and vector $\vec{u}$ obtained at the last iteration correspond to a bounded polynomial ODE and a Lyapunov function certifying boundedness.  
\section{Numerical examples}\label{s:numerics}

We now illustrate our \methodname\ framework on a two-dimensional system with no absorbing ellipsoid, 109 chaotic ODEs from the \texttt{dysts} library \cite{gilpin2021chaos}, and a reduced-order modeling problem for a reaction-diffusion PDE. Our implementation is in MATLAB and uses YALMIP \cite{Lofberg2004} to set up optimization problems, MOSEK v11.1.11 \cite{mosek} to solve them, and ChebFun \cite{Driscoll2014} to handle polynomials in the multivariate Chebyshev basis, which we use for numerical stability. Our code is freely available from \url{https://github.com/giofantuzzi/silas}.

\subsection{A bounded system with no absorbing ellipsoid}\label{ss:ex3}

As a first example, we apply \methodname\ to the two-dimensional ODE $\dot{\vec{x}} = \vec{f}^*(\vec{x})$ with $\vec{x} = (x_1,x_2)$ and
\begin{equation}\label{eq:ex3}
f^*(\vec{x}) =
\begin{pmatrix}
x_1 - x_2 - x_1 x_2^{2} + x_2^{3} \\
x_1 + x_2 - x_1^{2} x_2 - x_1^{3}
\end{pmatrix}.
\end{equation}
\Cref{fig:ex3_trainingData} shows that all solutions are bounded and converge either to a limit cycle or to one of five fixed points at $(0,0)$ and $(\pm1, \pm1)$. As we prove in \cref{app:analyticalEx3}, boundedness cannot be proven using quadratic Lyapunov functions, but quartic ones suffice. For example, the quartic polynomial Lyapunov function
\begin{equation}\label{eq:quarticLyap}
v(x_1,x_2)
=
x_1^2+x_2^2+\frac13 \left( x_1^4 + x_1^3 x_2 - x_1 x_2^3 + x_2^4\right)
\end{equation}
is coercive and satisfies the Lyapunov inequality \cref{e:lyapunov-inequality} with $\alpha = 1$ and $b = 15$. The corresponding absorbing set is shown in \Cref{fig:ex3_trainingData}.
\begin{figure}[t]
    \centering
    \includegraphics[width=0.3\linewidth]{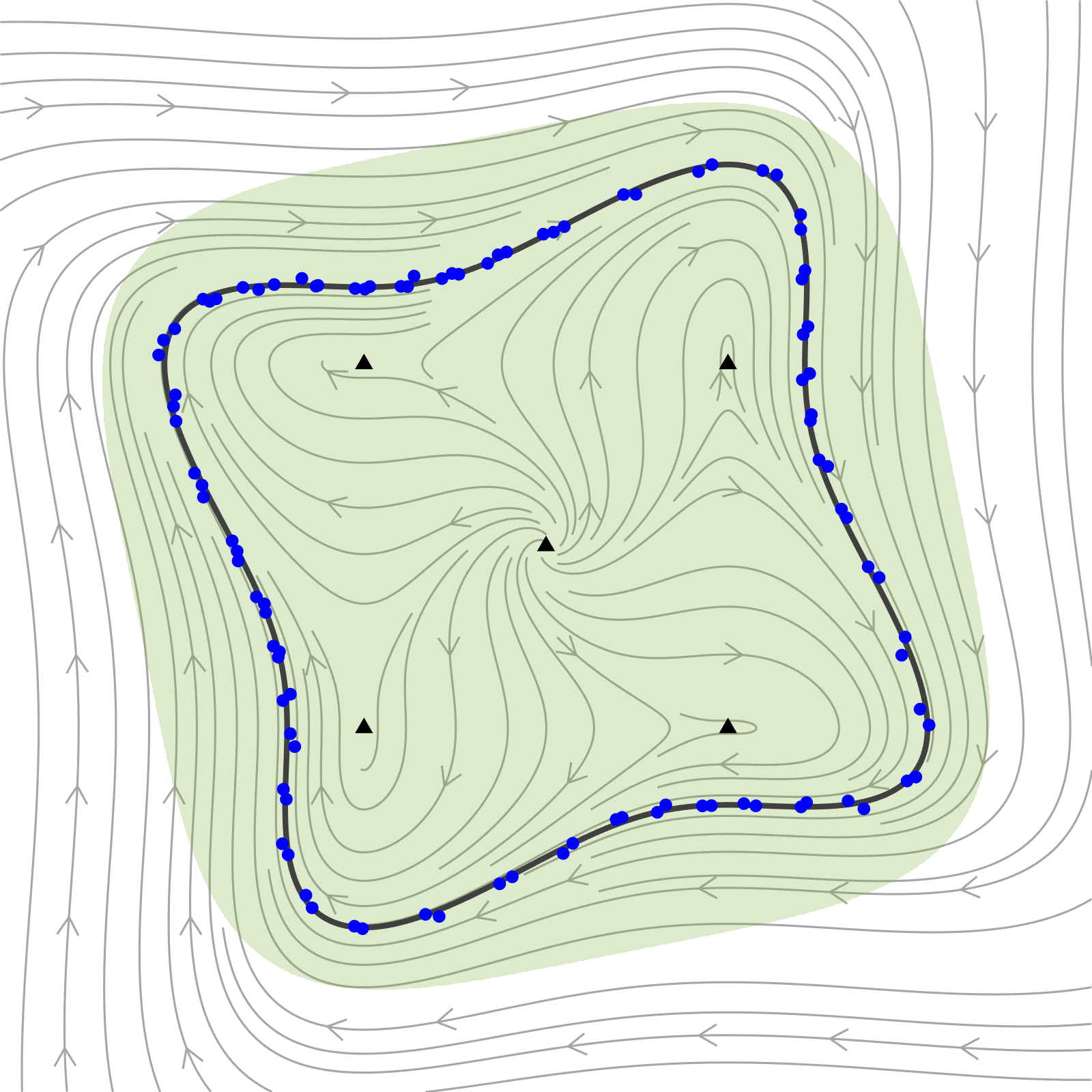}
    \caption{State-space representation of the two-dimensional cubic system from \cref{ss:ex3}. The figure shows the vector field $\vec{f}^*$ (gray arrows), the limit cycle (solid black line), the equilibria (black triangles), the absorbing set defined by \Cref{eq:quarticLyap} with $\alpha = 1$ and $b = 15$ (green shaded region), and the noisy state measurements used for model discovery (blue dots).}
    \label{fig:ex3_trainingData}
\end{figure}

\subsubsection{Data generation}
To generate a dataset for \methodname, we simulate the system until it settles on the limit cycle, then measure its state 100 times at intervals of 0.101 time units. Measurement errors are simulated by adding to each state measurement a normally distributed random perturbation with zero mean and with variance chosen to make the expected error magnitude 1\% of the exact measurement. These noisy measurements $\vec{x}_i$, shown as blue dots in \Cref{fig:ex3_trainingData}, are used to recover rate-of-change measurements $\vec{y}_i$ through an 8th-order central finite-difference formula. The final dataset comprises $m=92$ measurement pairs $(x_i, y_i)$.

\subsubsection{Results}
We apply \cref{alg:learn-ode} to the dataset described above for different ODE degrees $d_f$ and Lyapunov function degrees $d_v$. The hyperparameters are set to $\alpha=1$, $\beta=100$, $\kappa = 0.1$, $\varepsilon_1=0.1$, $\varepsilon_2=\varepsilon_3=0.01$, $\varepsilon_4=10^{-6}$, $\mat{\Lambda}=\frac{2}{5}I$,  $\vec{\mu}=0$ and $K = 5$.

\begin{figure}[t]
    \centering
    \includegraphics[width = \linewidth]{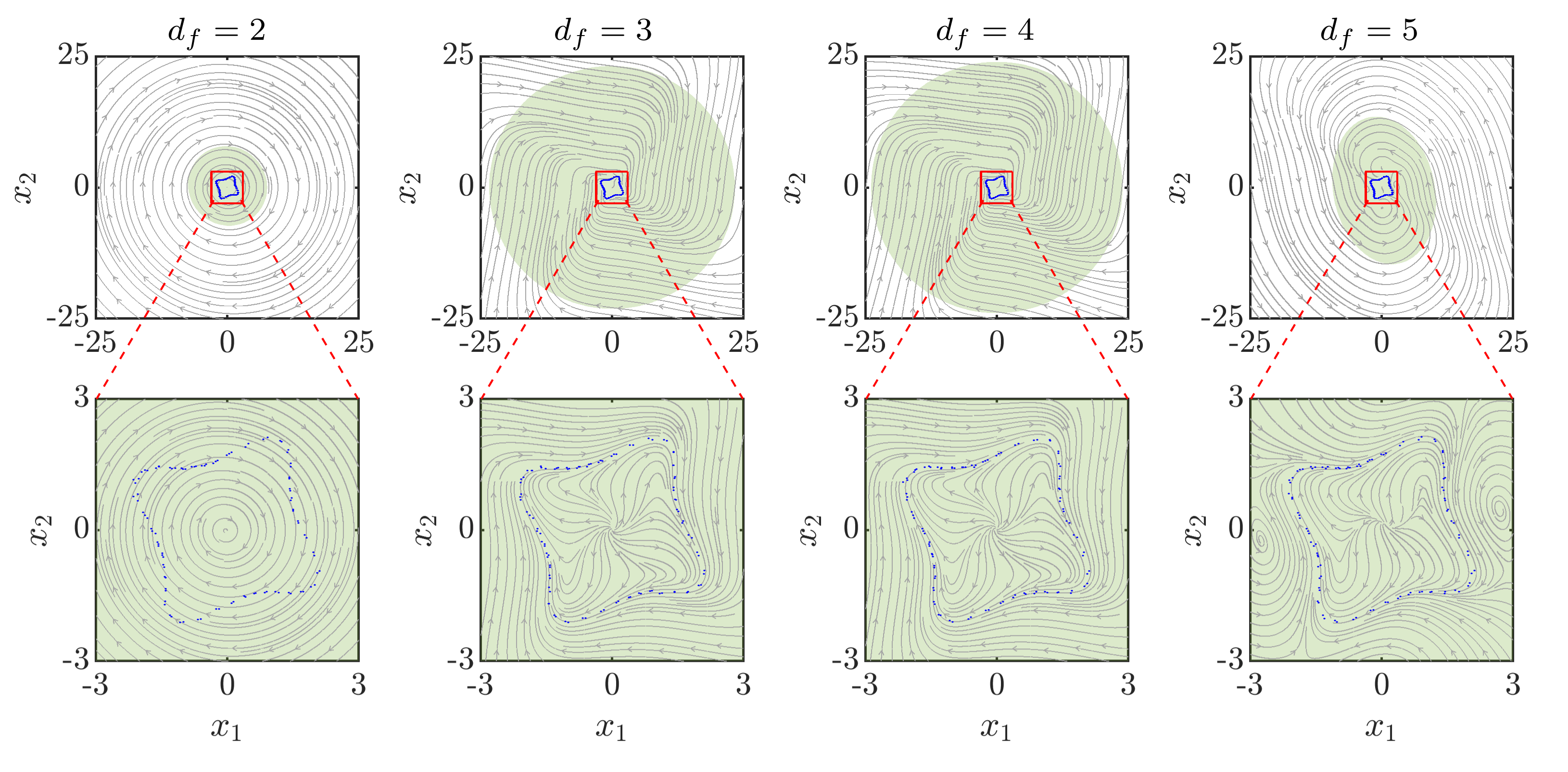}
    \vspace*{-25pt}
    \caption{Learned vector fields and corresponding absorbing sets (shaded in green) for Lyapunov function degree $d_v = 2$ and model degrees $d_f = 2,3,4,5$. Panels in the bottom row show a zoomed-in view of the box $[-3,3]^2$. \label{fig:cubic_dv2_results}}
\end{figure}

\begin{figure}[t]
    \centering
    \includegraphics[width = \linewidth]{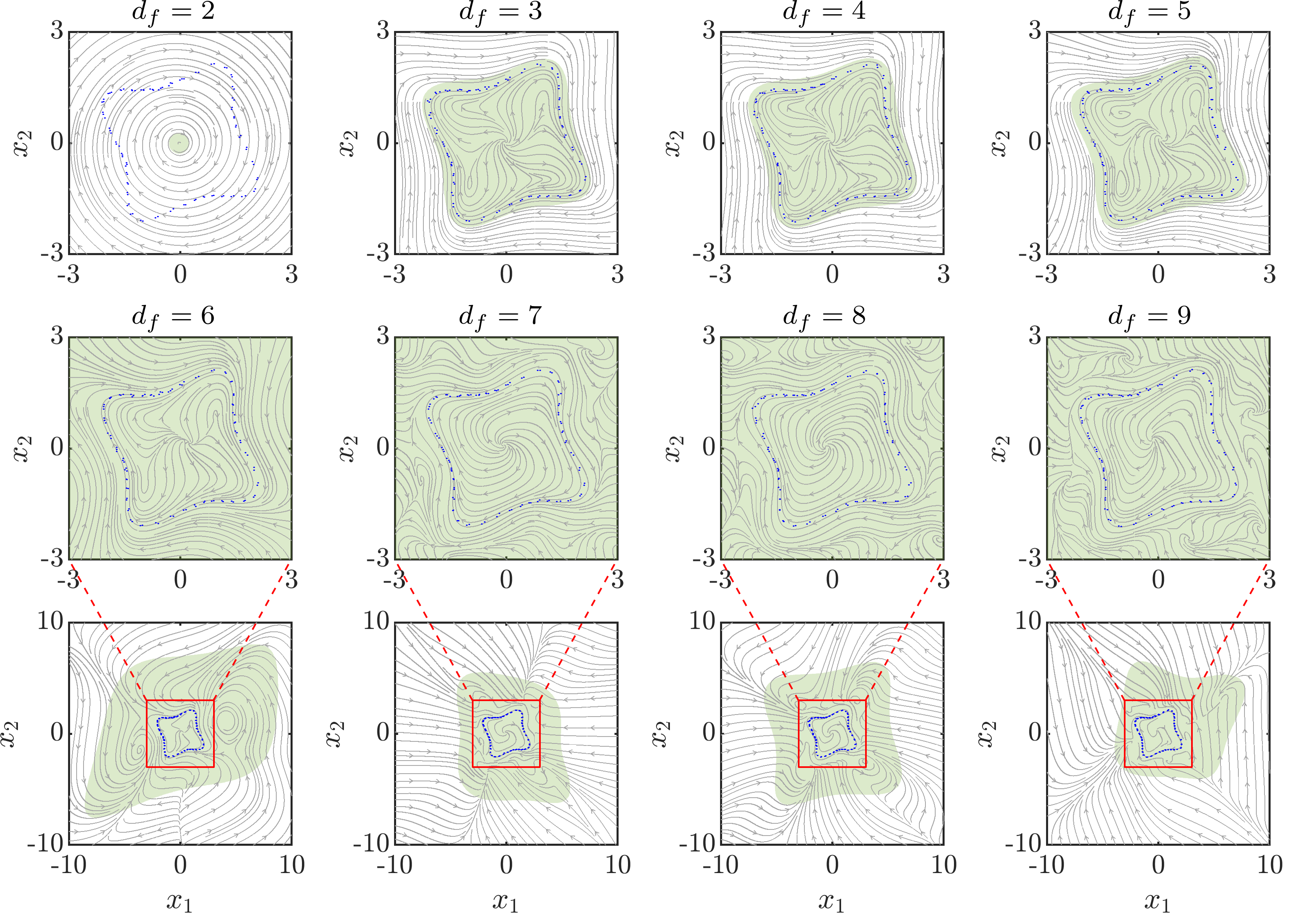}
    \vspace*{-20pt}
    \caption{Learned vector fields and corresponding absorbing sets (shaded in green) for Lyapunov function degree $d_v = 4$ and model degrees $d_f = 2,\ldots,9$. The bottom row shows zoomed-out views of the panels in the middle row. \label{fig:cubic_dv4_results}}
\end{figure}

\begin{figure}[t]
    \centering
    \includegraphics[width=\linewidth]{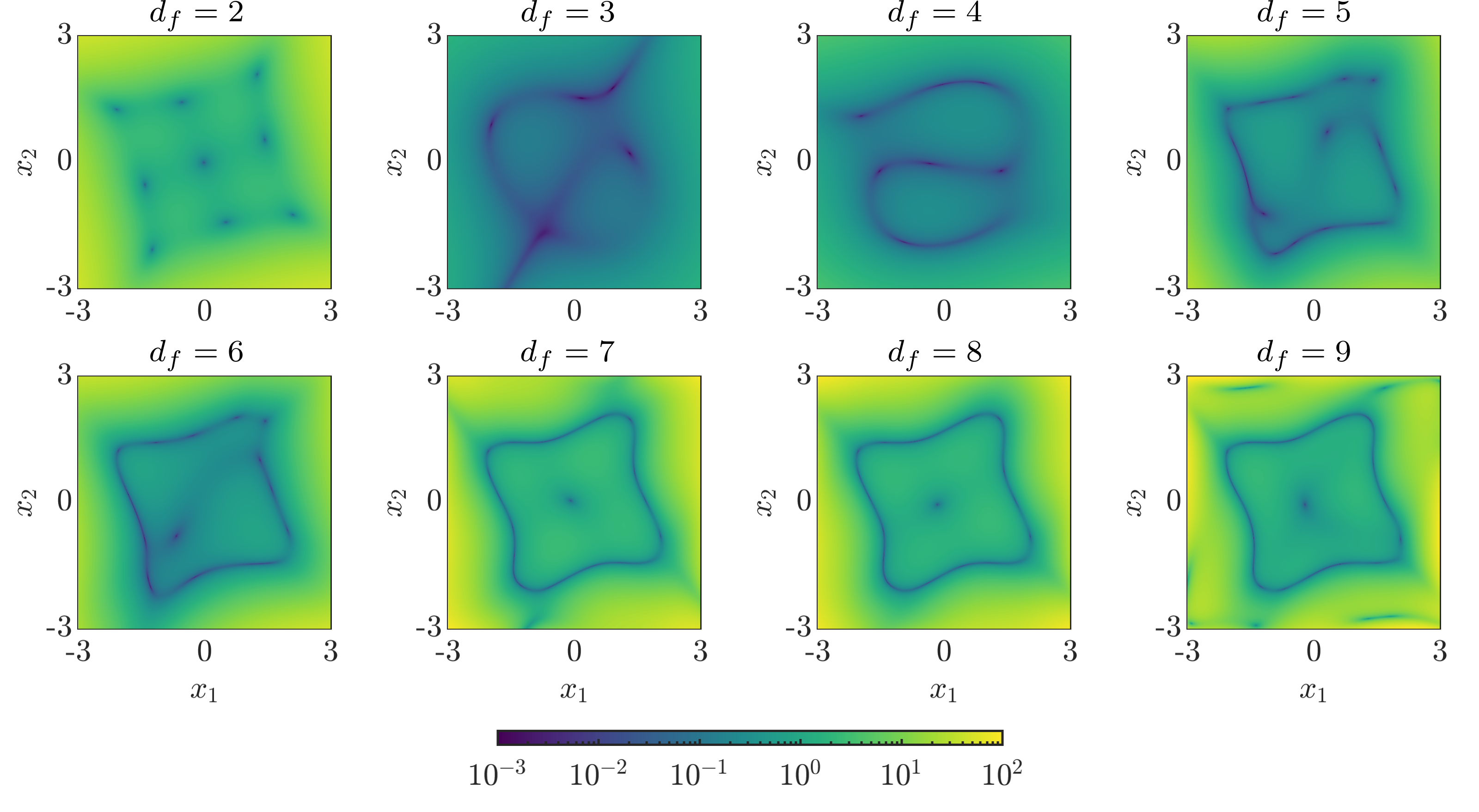}
    \vspace*{-20pt}
    \caption{Pointwise $2$-norm errors between the exact vector field of the cubic system in \cref{ss:ex3} and the vector fields learned with Lyapunov function degree $d_v = 4$ and model degree $d_f = 2,\ldots,9$.}
    \label{fig:cubic-err-dv4}
\end{figure}

We first fix $d_v=2$. The learned vector fields and certified absorbing sets for $d_f=2,3,4,5$ are shown in \Cref{fig:cubic_dv2_results}. As expected, $d_f=2$ is not expressive enough to capture the dynamics, yielding a system with a globally stable equilibrium at the origin rather than a limit cycle. The model for $d_f=3$, instead, has an attracting limit cycle and a certified absorbing ellipsoid. Thus, while the true system does not admit a quadratic Lyapunov function, it can be approximated by one that does, which \methodname\ recovers. Increasing $d_f$ further to $4$ and $5$ results in bounded ODE models with absorbing ellipsoids of varying size that accurately capture the true limit cycle dynamics, but may behave differently in other regions of the state space. This is expected because the training dataset reflects only the limit cycle dynamics. 

Next, we learn models with $d_v=4$ and $d_f=2,\dots,9$. The learned vector fields $f$ and absorbing sets are shown in \Cref{fig:cubic_dv4_results}, while the pointwise $2$-norm error $\|\vec{f}(x) - \vec{f}^*(x)\|_{2}$ is shown in \Cref{fig:cubic-err-dv4}. The model with the smallest maximum pointwise error is obtained for $d_f=3$. This model recovers well the limit cycle dynamics, generalizes well away from it, and its absorbing set is very tight around the attracting limit cycle. Increasing $d_f$ further does not affect the reconstruction of the limit cycle in any significant way, but reduces accuracy away from it. Again, this is expected because the training dataset contains only measurements from the limit cycle.

Finally, for both choices of the Lyapunov function degree $d_v$, the learned ODEs of degree $d_f=2k-1$ and $d_f=2k$ behave very similarly. This is a signature of the need for `lossless nonlinearities' in even-degree models (cf. \cref{rem:lossless-nnl}). Nevertheless, we stress that our even-degree ODEs do \emph{not} necessarily reduce to odd-degree ones of lower degree because the extra terms need not be set to zero by the optimizer. For example, the degree-$6$ ODE we obtained with $d_v=4$ contains nontrivial sextic terms, which we do not report for brevity. In contrast, the coefficients of degree-$4$ terms in our quartic ODE are all smaller than $10^{-7}$ and are at least three orders of magnitude smaller than other coefficients, suggesting that we discovered a cubic ODE.

\subsection{Chaotic ODEs from the \texttt{dysts} library}

Next, we discover bounded polynomial ODEs for the 109 autonomous chaotic systems in the \texttt{dysts} library \cite{gilpin2021chaos} with dimension up to $n = 6$ and no reported unbounded variables. Of these 109 ODEs, 82 have a polynomial vector field. Some of them, such as the Lorenz system, are known to admit absorbing sets. Others, such as the \texttt{SprottA} and \texttt{SprottB} systems, have bounded chaotic attractors but exhibit blowup for certain initial conditions. Nevertheless, we can apply \methodname\ to learn bounded models for the attractor dynamics. 

\subsubsection{Data generation}

We generate data for \methodname\ using the simulation tools provided with the \texttt{dysts} library. We simulate each of the 109 systems for 10 times the system's characteristic time $T$ reported by the library and starting from the default initial condition. We collect state and rate-of-change measurements with a fixed timestep equal to the minimum of $0.25$, $T/500$, or 20 times the default timestep reported by the library. These exact measurements are used to test the accuracy of the ODEs discovered by \methodname. In the training dataset $\mathcal{D}$, instead, we keep only the first half of the state measurements $\vec{x}_i$ from the simulation and estimate the corresponding rate-of-change values $\vec{y}_i$ with an order-8 centered difference scheme.

\subsubsection{Hyperparameter choices}\label{ss:dysts-hyperparams}
All results in this section were obtained by fixing the hyperparameters in \cref{alg:learn-ode} to
$\alpha = 0.1$, 
$\beta = 10^6$, 
$\kappa = 0.1$,
$\varepsilon_1 = \varepsilon_2 = \varepsilon_3 = 0.01$, 
$\varepsilon_4 = 10^{-6}$, and
$K=1$.
The matrix $\mat{\Lambda}$ and the vector $\vec{\mu}$ were adapted to each system by fixing $\mat{\Lambda} = 0.1 \mat{\Lambda}_0$ and $\vec{\mu} = 0.1 \vec{\mu}_0$, where $\mat{\Lambda}_0$ and $\vec{\mu}_0$ are selected such that the function $\vec{z}\mapsto \mat{\Lambda}_0^{-1}(\vec{z} - \vec{\mu}_0)$ maps the hypercube $[-1,1]^n$ to a box bounding the training data with a 25\% buffer. Precisely, if  $\vec{x}^{\rm max}$ and $\vec{x}^{\rm min}$ are vectors listing the maximum and minimum values of each state variable over the dataset, then $\mat{\Lambda}_0$ is the diagonal matrix
\begin{equation*}
    \mat{\Lambda}_0 = \begin{pmatrix}
        \frac{8}{4 (x^{\rm max}_1 - x^{\rm min}_1) + |x^{\rm max}_1| + |x^{\rm min}_1|}
        \\
        &\ddots \\
        &&
        \frac{8}{4 (x^{\rm max}_n - x^{\rm min}_n) + |x^{\rm max}_n| + |x^{\rm min}_n|}
    \end{pmatrix}
\end{equation*}
while $\mu_0$ is defined as
\begin{equation*}
    \vec{\mu}_0 = \begin{pmatrix}
        -\frac{4 (x^{\rm max}_1 + x^{\rm min}_1) + |x^{\rm max}_1| - |x^{\rm min}_1|}{4 (x^{\rm max}_1 - x^{\rm min}_1) + |x^{\rm max}_1| + |x^{\rm min}_1|}
        \\
        \vdots\\
        -\frac{4 (x^{\rm max}_n + x^{\rm min}_n) + |x^{\rm max}_n| - |x^{\rm min}_n|}{4 (x^{\rm max}_n - x^{\rm min}_n) + |x^{\rm max}_n| + |x^{\rm min}_n|}
    \end{pmatrix}.
\end{equation*}
For numerical stability, all polynomial models and Lyapunov functions are expressed using multivariate Chebyshev bases in the scaled state vector $\mat{\Lambda}_0 \vec{x} + \vec{\mu}_0$.
\begin{figure}[t]
    \centering
    \includegraphics[width=\linewidth]{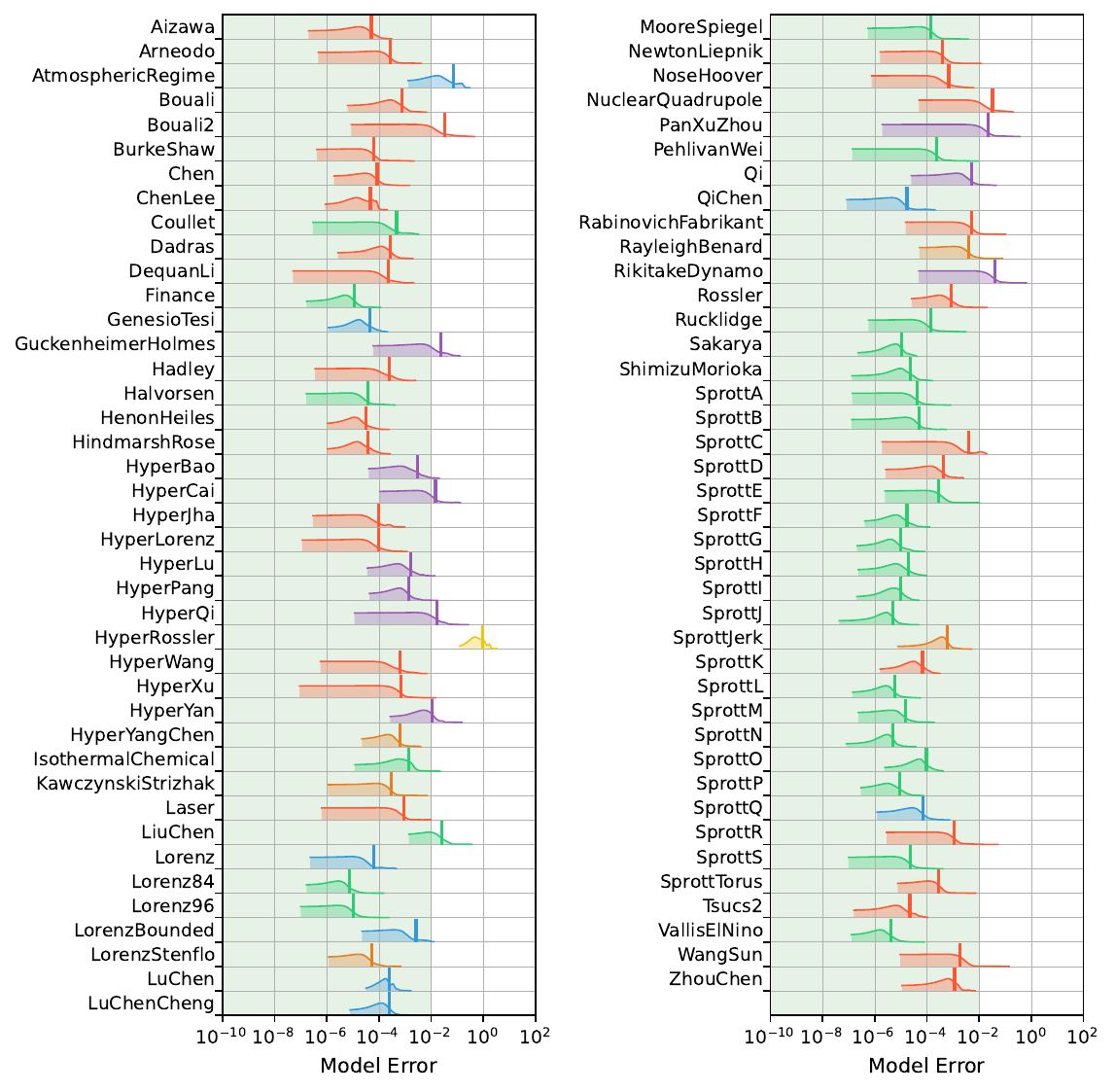}
    \vspace*{-20pt}
    \caption{Error distributions (curves) and mean errors (vertical bars) for 82 polynomial ODEs in the \texttt{dysts} library. Colors identify models of degree 
    2~({\color[HTML]{f1c40f}{$\blacksquare$}}),
    3~({\color[HTML]{9b59b6}{$\blacksquare$}}),
    4~({\color[HTML]{e67e22}{$\blacksquare$}}),
    5~({\color[HTML]{ff5733}{$\blacksquare$}}),
    6~({\color[HTML]{3498db}{$\blacksquare$}}), and
    7~({\color[HTML]{2ecc71}{$\blacksquare$}}). Green shading indicates errors smaller than~1\%.}
    \label{f:dysts-poly}
\end{figure}

\subsubsection{Results}

For each of the 109 systems in the library, we discover bounded polynomial ODE models of degree $2$ to $7$ with Lyapunov functions of degree $2$ and $4$. We then compute the relative error 
$\|\vec{f}^*(x) - \vec{f}(x)\|_2 / \|\vec{f}^*(\vec{x})\|_2$ between the exact vector field $\vec{f}^*$ and the discovered vector field $\vec{f}$ at the test points $\vec{x}$. Error distributions for the 109 models with the smallest mean error are shown in \Cref{f:dysts-poly,f:dysts-nonpoly}, where the mean error is shown as a vertical bar and colors indicate the model degree. 
The mean error is smaller than 1\% for 71 of the 82 polynomial systems and for 10 of the 27 non-polynomial ones. Thus, \methodname\ works robustly for systems governed by polynomial equations, and can be accurate for non-polynomial dynamics if their nonlinearities can be approximated well by low-degree polynomials.

\begin{figure}[t]
    \centering
    \includegraphics[width=\linewidth]{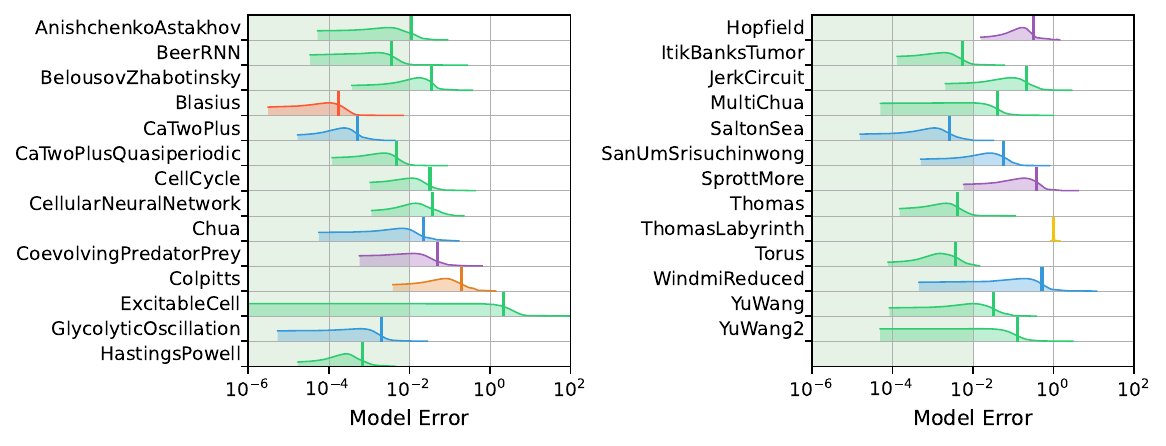}
    \caption{Error distributions (curves) and mean errors (vertical bars) for 27 non-polynomial ODEs in the \texttt{dysts} library. Colors identify models of degree 
    2~({\color[HTML]{f1c40f}{$\blacksquare$}}),
    3~({\color[HTML]{9b59b6}{$\blacksquare$}}),
    4~({\color[HTML]{e67e22}{$\blacksquare$}}),
    5~({\color[HTML]{ff5733}{$\blacksquare$}}),
    6~({\color[HTML]{3498db}{$\blacksquare$}}), and
    7~({\color[HTML]{2ecc71}{$\blacksquare$}}). Green shading indicates errors smaller than~1\%.}
    \label{f:dysts-nonpoly}
\end{figure}
\begin{figure}[t]
    \centering
    \includegraphics[width=\linewidth]{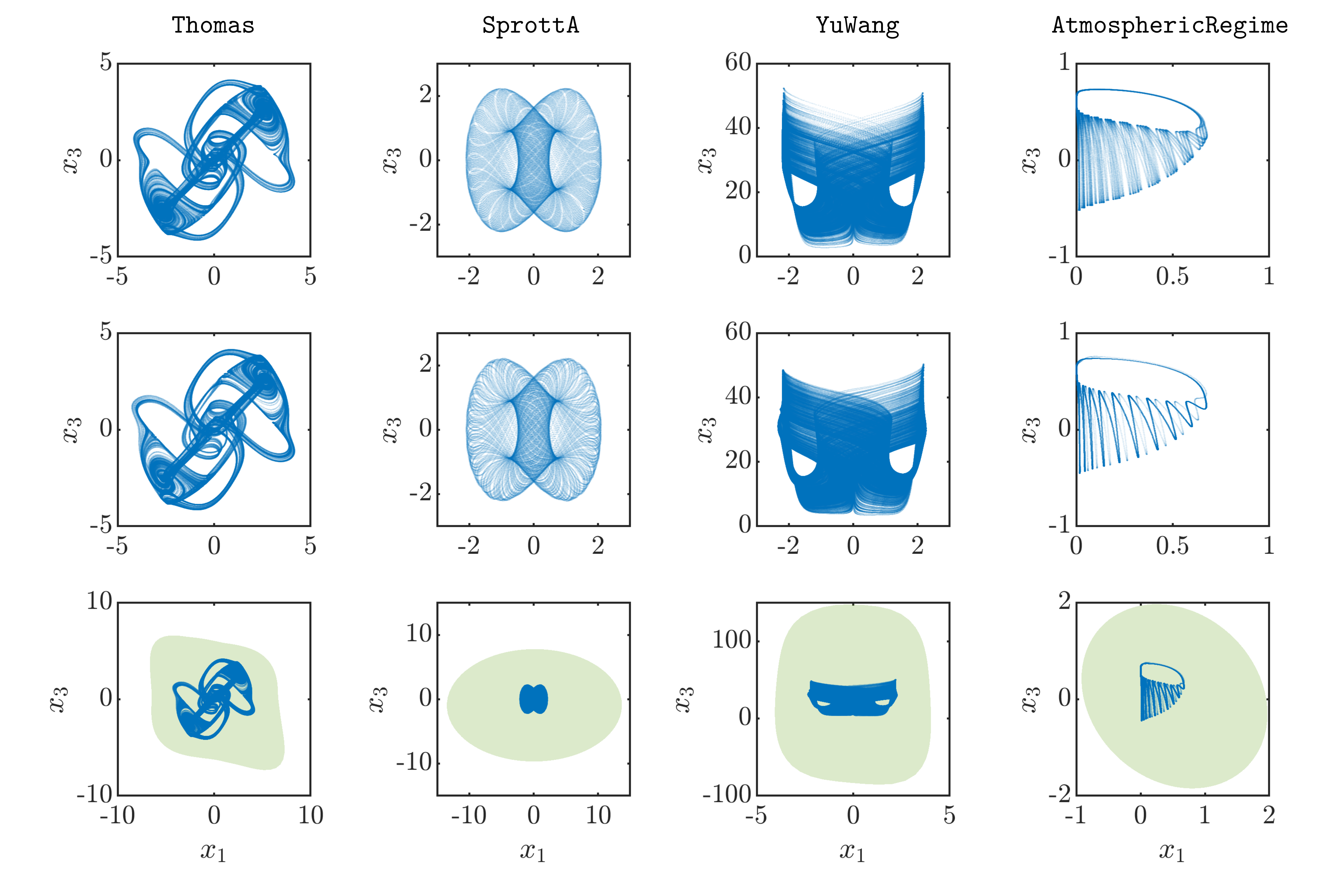}
    \vskip-15pt
    \caption{
    Attractors for selected systems from the \texttt{dysts} library. \emph{Top row:} The true attractor. \emph{Middle row:} Discovered attractors. \emph{Bottom row:} Discovered attractors and absorbing sets certifying the boundedness of the discovered ODEs.}
    \label{f:attractors}
\end{figure}

\Cref{f:attractors} compares the exact and discovered attractors for a non-polynomial system with a polynomial model with less than 1\% mean error (\texttt{Thomas}), a non-polynomial system with a polynomial model with 4\% mean error (\texttt{YuWang}), a polynomial system with no absorbing set which we approximate within 1\% with a bounded ODE (\texttt{SprottA}), and a polynomial system for which our best model achieves only a 7\% mean error (\texttt{AtmosphericRegime}). Absorbing sets guaranteeing the boundedness of the discovered models are also shown. In all cases, the discovered attractors are topologically similar to the exact ones, showing that \methodname\ can work well even when the true system dynamics are not governed by bounded polynomial ODEs. Of course, there are also cases in which our bounded ODE models are unable to accurately reproduce the true system dynamics, such as the non-polynomial systems \texttt{SprottMore}, \texttt{Hopfield}, and \texttt{WindmiReduced} (plots not included for brevity). Poor results for these systems, however, are expected because the \texttt{SprottMore} system is governed by discontinuous equations involving the sign function, while the \texttt{Hopfield} and \texttt{WindmiReduced} ODEs include hyperbolic tangent terms with sharp transitions. Neither of these nonlinearities is approximated well by degree-$7$ polynomials.

\begin{figure}[t]
    \centering
    \includegraphics[width=\linewidth]{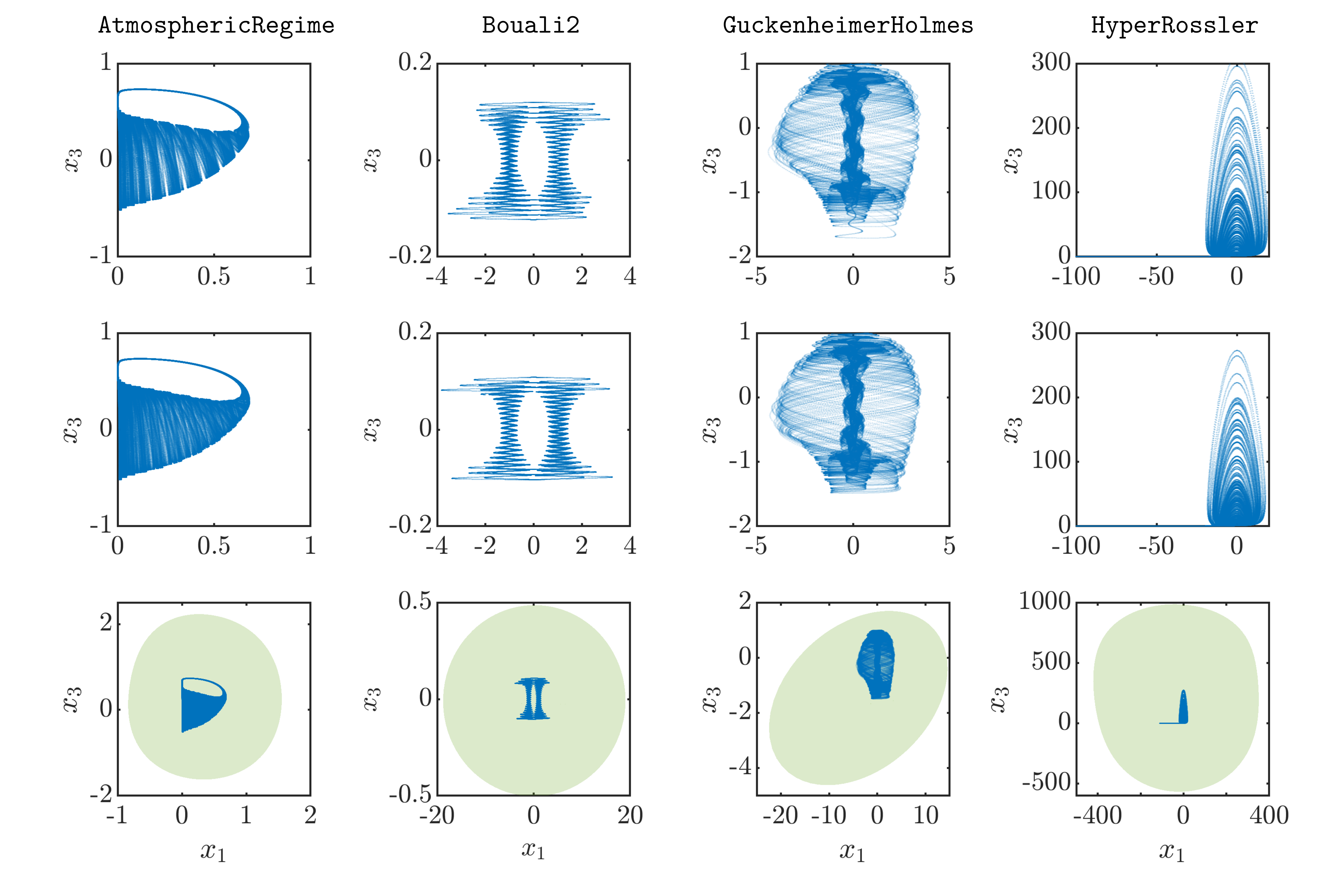}
    \vskip-15pt
    \caption{
    Attractors for selected systems from the \texttt{dysts} library, discovered with a refined dataset. \emph{Top row:} True attractors. \emph{Middle row:} Discovered attractors. \emph{Bottom row:} Discovered attractors and absorbing sets certifying the boundedness of the discovered ODEs.}
    \label{f:attractors-improved}
\end{figure}

Finally, we show that more accurate models for many systems exhibiting large errors in \Cref{f:dysts-poly} can be obtained by extending the dataset or further tuning the hyperparameters of \methodname. For example, the mean model error for the \texttt{AtmosphericRegime} system can be reduced from 7\% to 0.07\% if we refine the dataset by taking five times the amount of sequential state measurements at a reduced timestep of $0.1$ time units. As shown in \Cref{f:attractors-improved}, this two-order-of-magnitude improvement translates into a better qualitative approximation of the system's chaotic attractor with little effect on the absorbing set certifying the boundedness of the discovered model. 
Enlarging the dataset and increasing the data sampling rate leads to similar improvements for the polynomial systems \texttt{Bouali2}, \texttt{HyperRossler}, and \texttt{GuckenheimerHolmes}, resulting in provably bounded ODE models that better approximate the true chaotic attractors (see \Cref{f:attractors-improved}). We expect similar improvements for other systems, too, but a detailed system-by-system investigation is not within the scope of the present example. 

\subsection{Reduced-order models of a reaction-diffusion PDE} \label{ss:reacDiff-pde}

As our last example, we construct a reduced-order model for a one dimensional PDE model of a two-species reaction-diffusion system with intermittent bursting \cite{PhysRevLett.68.714}. The PDEs are
\begin{subequations}\label{e:rd-pde}
    \begin{align}
    r_t &= D r_{\xi\xi} + 100\bigl(s - r^2 - r^3\bigr) \label{eq:pde-u}\\
    s_t &= D s_{\xi\xi} - r + 0.01 \label{eq:pde-v}
    \end{align}
\end{subequations}
with spatial variable $\xi\in[-1,1]$, where subscripts denote partial derivatives. The solutions $r(\xi, t)$ and $s(\xi, t)$ are subject to the Dirichlet boundary conditions $r(\pm 1, t)=-2$ and $s(\pm 1, t)=-4$. Following \cite{PhysRevLett.68.714}, we fix $D = 0.1289228$. 

\subsubsection{Data generation}

To generate data, we simulate \cref{e:rd-pde} with \texttt{Dedalus} \cite{burns2020dedalus}, which uses a pseudospectral collocation method for discretization in space. We set the spatial resolution to $N_x=128$ collocation points and integrate the discretized PDE in time with the built-in Runge--Kutta scheme.
We discard the solution until time $100$ to ensure the dynamics have settled into a weakly chaotic regime with intermittent bursting. We then collect spatial `snapshots' $r(\xi,t_j)$ and $s(\xi,t_j)$ of the PDE states at times $t_j$ separated by $0.005$ time units, which will be used for training and testing. 

\subsubsection{Reduced-order modeling strategy}

To build a reduced-order model for the PDE \cref{e:rd-pde}, we first approximate
\begin{equation}\label{e:pod-expansion}
r(\xi,t) \approx \bar r(\xi) + \sum_{i=1}^n x_i(t) \phi_i(\xi)
\qquad\text{and}\qquad
s(\xi,t) \approx \bar s(\xi) + \sum_{i=1}^n x_i(t) \psi_i(\xi),
\end{equation}
and then learn an ODE $\dot{\vec{x}}=\vec{f}(\vec{x})$ for the vector $\vec{x} = (x_1,\ldots,x_n)$ of approximation coefficients.
In the approximation \cref{e:pod-expansion}, $\bar{r}$ and $\bar{s}$ are the time averages of the snapshots in our dataset, while $\phi_1,\ldots,\phi_n$ and $\psi_1,\ldots,\psi_n$ are the $n$ principal basis functions obtained by applying the proper orthogonal decomposition (POD) \cite{Holmes1996} to our snapshot dataset after removing the temporal mean. Projecting the mean-free snapshots onto these basis functions (hereafter called `POD modes') yields measurements $\vec{x}(t_j)$ at the sampling times $t_j$, which we use to recover approximate rate-of-change measurements $y_j \approx \dot{\vec{x}}(t_j)$ through an order-8 centered difference scheme. The resulting pairs $(\vec{x}(t_j), y_j)$ constitute the dataset for \methodname.
\begin{figure}[t]
    \centering
    \includegraphics[width=0.8\linewidth]{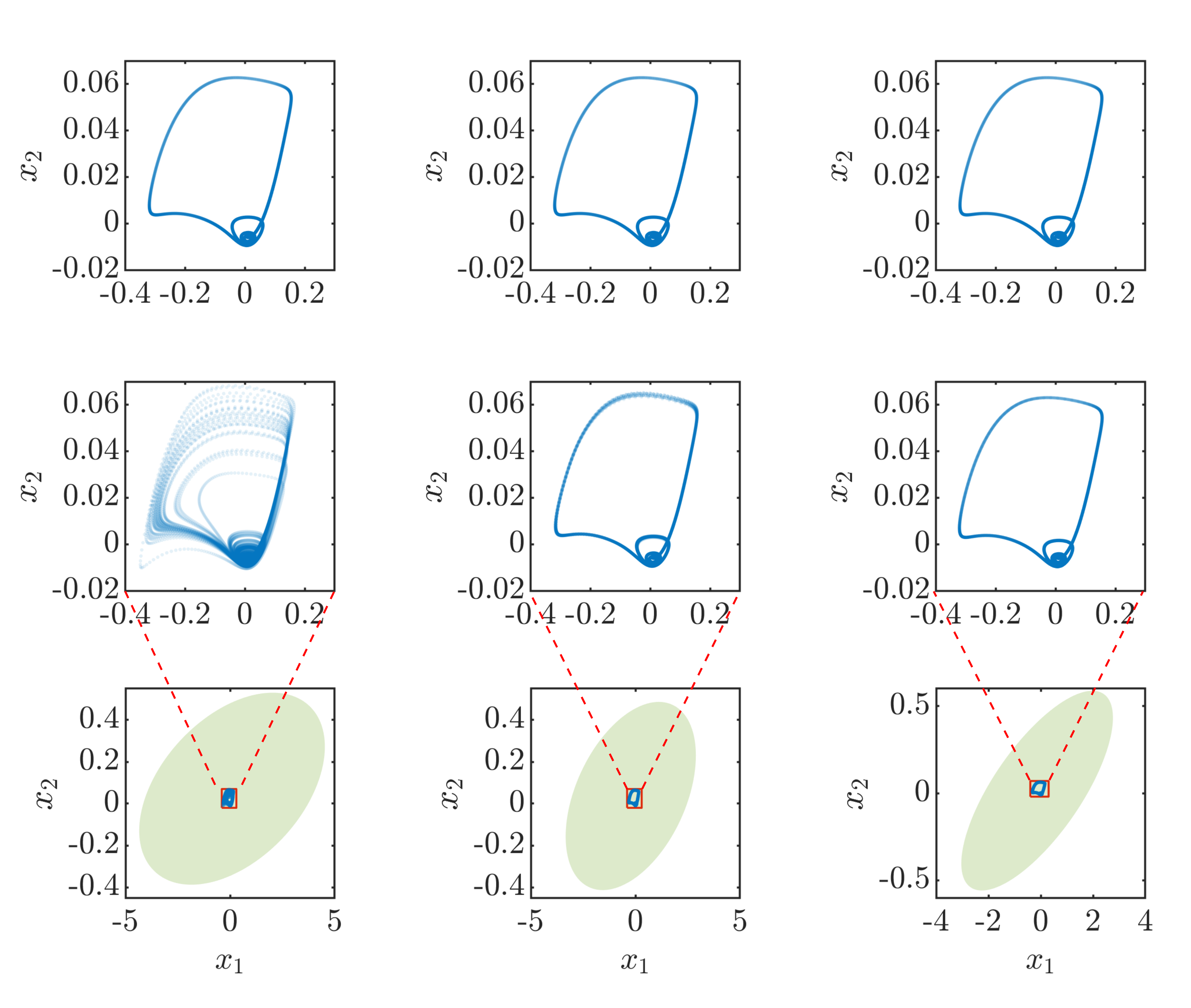}
    \caption{Modeled vs. true POD-coefficient dynamics for the PDE example in \Cref{ss:reacDiff-pde}, projected onto the $(x_1,x_2)$ plane. The reduced-order models are for $n=3$ (left), $4$ (middle), and $5$ (right) POD modes. Panels in each column show the PDE dynamics (top), the reduced-order ODE dynamics (middle), and the reduced-order ODE dynamics with their absorbing ellipsoid (bottom).} 
    \label{fig:reacDiff_trapping}
\end{figure}

\subsubsection{Results}

We use \methodname\ to discover bounded polynomial ODE models $\dot{\vec{x}}=\vec{f}(\vec{x})$ for the PDE dynamics using $n=3$, $n=4$ and $n=5$ POD modes. 
In all cases, the hyperparameters in \cref{alg:learn-ode} are fixed to 
$\alpha = 1$, 
$\beta = 10^6$, 
$\kappa = 10$, 
$\varepsilon_1 = \varepsilon_2 = \varepsilon_3 = 0.01$, 
$\varepsilon_4 = 10^{-6}$, and
$K=1$.
The matrix $\mat{\Lambda}$ and the vector $\vec{\mu}$ are selected with the strategy described in \cref{ss:dysts-hyperparams}. We vary the polynomial degrees in the range $d_f\in\{2,\ldots,7\}$ while keeping the Lyapunov function degree fixed to $d_v=2$. This choice is justified because, as shown in \cref{app:pde-boundedness}, the PDE \cref{e:rd-pde} has an absorbing ellipsoid in the Lebesgue space $L^2(-1,1)\times L^2(-1,1)$. 
For $n = 3$ and $n = 4$, we use $25\,000$ snapshots for training and $20\,000$ snapshots for testing. For $n = 5$, instead, we use 10 times more data for the training set and reduce the hyperparameter $\beta$  to $100$ (models trained on a smaller dataset do not reproduce the intermittent bursting for our hyperparameter choices).
\begin{figure}[t]
    \centering
    \includegraphics[width=0.95\linewidth]{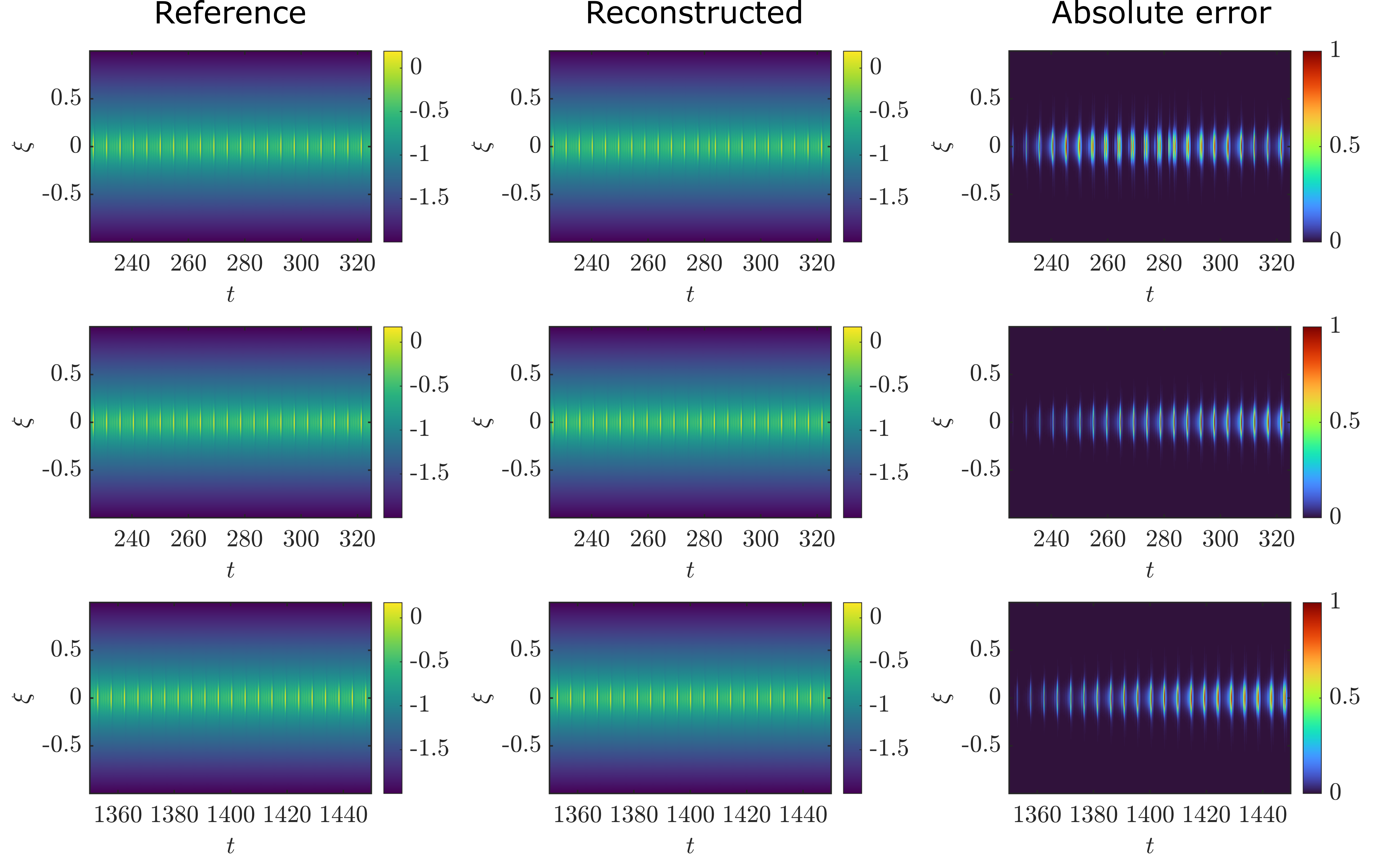} 
    \caption{Reconstructed vs.\ true $r$-component of the PDE solution in \Cref{ss:reacDiff-pde}, obtained from reduced-order models with $n=3$ (top), $4$ (middle), and $5$ (bottom) POD modes. In each row, panels show the reference PDE simulation (left), the reconstruction from the learned ODE model (middle), and the absolute error between the two (right).}
    \label{fig:reacDiff_u}
\end{figure}
The ODEs that best reproduce the intermittent dynamics of \Cref{e:rd-pde} are obtained with $d_f = 5$ for $n=3$, with $d_f=5$ for $n=4$, and with $d_f=3$ for $n=5$. \Cref{fig:reacDiff_trapping} compares trajectories of these ODEs to projections of the PDE dynamics onto the POD modes. The absorbing ellipsoid certifying the boundedness of the discovered ODE models is also shown for illustration. Clearly, there is a good qualitative agreement between the modeled and the true dynamics throughout the test window. This conclusion is corroborated by \Cref{fig:reacDiff_u}, where space-time approximations of the PDE state $r(\xi, t)$ recovered from our reduced-order ODE model are compared to the original PDE simulation data (Results for the PDE state $s(\xi, t)$ are similar and are not reported for brevity.) The agreement between reduced-order model and simulations is best with $n=4$ POD modes, but in all cases our bounded models capture the intermittent bursting of the PDE dynamics for the entire test window ($100$ time units). Thus, \methodname\ is able to produce provably bounded reduced-order models that reproduce the PDE's nontrivial attractor dynamics over relatively long time horizons. 
\section{Conclusion}\label{s:conclusion}

We presented \methodname, a data-driven framework for discovering polynomial ODEs with provably bounded trajectories in which boundedness is certified by polynomial Lyapunov functions characterizing compact absorbing sets. Concretely, \methodname\ identifies bounded polynomial ODEs and the associated Lyapunov functions by finding a feasible solution of a well-posed nonconvex optimization problem. The computations rely on three key ingredients: sum-of-squares techniques from polynomial optimization to enforce Lyapunov constraints, a block-coordinate scheme with convex subproblems that circumvents the nonconvex coupling between the ODE and Lyapunov parameters, and a novel initialization that infers a candidate Lyapunov function from data to ensure the feasibility of the block-coordinate scheme. Experiments on more than 100 nonlinear systems show that \methodname\ recovers accurate bounded ODE models for many low-dimensional systems with nontrivial dynamics, including some with non-polynomial governing equations or exhibiting blowup for selected initial conditions.

Our computational approach is highly flexible and allows for polynomial ODEs and Lyapunov functions of arbitrary degree. The main bottleneck is the solution of optimization problems with sum-of-squares polynomials, which have the same number of variables as the ODEs and whose degree increases with the ODE and Lyapunov function degree. For the examples in \cref{s:numerics}, which have up to six variables and sum-of-squares degree at most 12, computations complete within minutes on a laptop. The computational complexity, however, deteriorates quickly with state-space dimension and degree, so \methodname\ is currently not well-suited to discovering high-dimensional, high-degree ODEs. On the other hand, for high-dimensional systems with low-dimensional intrinsic dynamics, reduced-order models are within reach. In the PDE example of \cref{ss:reacDiff-pde} we found such models after projecting the PDE dynamics onto a set of POD modes. Further work should explore if better low-dimensional embeddings can be found with autoencoders, which have already been used to learn dynamics without boundedness constraints~\cite{Bramburger2021}. 

The solution of problem \cref{eq:monster} poses further computational questions beyond the issue of scalability to high-dimensional and high-degree polynomial dynamics. First, while our block-coordinate optimization scheme generates feasible iterates with non-increasing objective values, we have no proof that these iterates converge to a local minimizer, even along subsequences. We wonder if existing results for block-coordinate optimization algorithms (see, e.g., \cite{Bertsekas1999,Grippo1999,Beck2013} and references therein) can be adapted to \cref{eq:monster}. Second, it is worth investigating if this problem can be solved efficiently by recent algorithms for nonconvex optimization over sum-of-squares polynomials \cite{Cunis2025}.
The availability of computational methods that directly handle nonconvexity would not only ensure convergence to a locally optimal model, but would also allow for further extensions of \methodname. First, instead of penalizing errors in the ODE's vector field in the objective of~\cref{eq:monster}, one could penalize errors in the measured state through a `flow map' formulation, which avoids rate-of-change measurements and is less sensitive to measurement noise~\cite{GOYAL2025134893,Goyal2022}. Second, nonconvex optimization algorithms would enable us to discover bounded models for discrete-time dynamics. Indeed, the discrete-time setting does not allow for block-coordinate optimization schemes with convex subproblems because the discrete-time version of the Lyapunov inequality \cref{e:lyapunov-inequality} involves the composition $v \circ f$, which is nonconvex in the dynamics map $f$ even if the Lyapunov function $v$ is held fixed.

Finally, a natural question concerns the approximation capabilities of bounded polynomial ODEs endowed with polynomial Lyapunov functions. Specifically, can a smooth ODE with a compact absorbing set be approximated uniformly on that set by a polynomial ODE whose boundedness is certified by a polynomial Lyapunov function? If so, can our \methodname\ framework recover such approximations? Further, how do the approximation errors scale with the degree of the vector field, the degree of the Lyapunov polynomial, and the size and distribution of the training dataset? Addressing these questions would provide further theoretical foundations for \methodname\ and related bounded ODE discovery frameworks.

% Acknowledgments
\subsection*{Acknowledgments}
AA was funded by the European Union's Horizon Europe MSCA project ModConFlex (grant number 101073558). GF was partially supported by the DFG-ANR project MONET (DFG grant number 568735456). We thank Torbjørn Cunis for bringing recent methods for nonconvex SOS programming to our attention.

% Bibliography
\bibliographystyle{abbrvnat}
\bibliography{refs}

% Appendix
\appendix
\crefalias{section}{appendix}
\crefalias{subsection}{appendix}
\section{Absorbing sets for the ODE in \texorpdfstring{\cref{ss:ex3}}{Section \ref{ss:ex3}}}
\label{app:analyticalEx3}
We show that the two-dimensional ODE $\dot{\vec{x}} = \vec{f}^*(\vec{x})$ with $f^*$ given by \cref{eq:ex3} has a compact absorbing set described by a quartic polynomial Lyapunov function, but no absorbing ellipsoid. To lighten the notation, we write $f$ instead of $f^*$.
\begin{proposition}
    If $f$ is the vector field in \cref{eq:ex3}, the polynomial
    $v(x_1,x_2)
    =
    x_1^2+x_2^2+\frac13 \left( x_1^4 + x_1^3 x_2 - x_1 x_2^3 + x_2^4\right)
    $
    satisfies the assumptions of \Cref{th:lyap-fun} with $\alpha=1$ and sufficiently large $b$.
\end{proposition}
\begin{remark}
    We do not determine the value of $b$ for brevity, but one can verify that $b=15$ works using the SOS optimization techniques described in \cref{ss:model-based-lyap}.
\end{remark}
\begin{proof}
    Using the inequality $\smash{|ab| \leq \frac14|a|^4 + \frac34|b|^{4/3}}$ we can estimate $x_1 x_2^3 \geq - \frac14 x_1^4 - \frac34 x_2^4$ and $x_2 x_1^3  \geq - \frac34 x_1^4 - \frac14 x_2^4$. We then conclude that $v(\vec{x})\geq x_1^2 + x_2^2$, so $v$ is coercive.
    
    To prove that $v$ satisfies the Lyapunov inequality \cref{e:lyapunov-inequality} for $\alpha=1$ and a suitably large $b$, it suffices to show that the leading-order homogeneous part of the polynomial $b -v-\ip{\vec{f}}{\nabla v}$ is coercive.
    A straightforward computation gives
    $$
    b -v - \ip{\vec{f}}{\nabla v} = \frac13 \left( x_1^6 + x_1^5 x_2 + x_1^4 x_2^2 + x_1^2 x_2^4 - x_1 x_2^5 + x_2^6 \right) + \text{lower-degree terms}.
    $$
    The degree-$6$ leading homogeneous part is seen to be coercive upon rewriting it as
    % \begin{equation*}
    %       \frac16 \, x_1^4  \left( x_1 + x_2 \right)^2 
    %     + \frac16 \, x_2^4 \left( x_1 - x_2 \right)^2
    %     + \frac16 \left( x_1^4 + x_2^4 \right) \left( x_1^2 + x_2^2 \right).
    % \end{equation*}
    $\frac16 \, x_1^4  \left( x_1 + x_2 \right)^2 
        + \frac16 \, x_2^4 \left( x_1 - x_2 \right)^2
        + \frac16 \left( x_1^4 + x_2^4 \right) \left( x_1^2 + x_2^2 \right).
    $
    The proof is complete.
\end{proof}
\begin{proposition}
    If $f$ is the vector field in \cref{eq:ex3}, there exist no $\alpha>0$, $b\in \R$, and coercive quadratic polynomial $v$ satisfying the Lyapunov inequality \cref{e:lyapunov-inequality}.
\end{proposition} 
\begin{proof}
    For the sake of contradiction, suppose that $v$ is a quadratic polynomial such that the Lyapunov inequality \cref{e:lyapunov-inequality} holds for some $\alpha>0$ and $b\in\R$.
    Since the vector field $\vec{f}$ in \cref{eq:ex3} is equivariant under the symmetry transformation $x\mapsto -x$, a standard symmetry reduction argument shows that $\frac12 v(\vec{x})+ \frac12v(-\vec{x})$ also satisfies \cref{e:lyapunov-inequality} for the same values of $\alpha$ and $b$. We may therefore assume without any loss of generality that $v(-\vec{x})=v(\vec{x})$, so it takes the form
    $v(x) = A x_1^2 + B x_1 x_2 + C x_2^2 + D$ 
    for some coefficients $A$, $B$, $C$ and $D$. We must also have $A,C>0$ because $v$ is coercive by assumption.
    
    Upon substituting this quadratic polynomial $v$ and the vector field $\vec{f}$ from \cref{eq:ex3} into the Lyapunov inequality \cref{e:lyapunov-inequality} we find that
    \begin{align*}
        0\leq 
        b - D
        - (A + B\alpha +2A\alpha)x_1^2 
        - (B+2B\alpha-2A\alpha+2C\alpha)x_1 x_2
        \\
        - (C - B\alpha +2C\alpha) x_2^2 
        + B\alpha x_1^4 
        + (B + 2C)\alpha x_1^3 x_2
        \\
        + 2(A+C)\alpha x_1^2 x_2^2
        + (B - 2A)\alpha x_1 x_2^3
        - B \alpha x_2^4
    \end{align*}
    for all $\vec{x}=(x_1,x_2)\in\R^2$.
    This means the coefficients of the monomials $x_1^4$ and $x_2^4$ must be nonnegative, so $B=0$. We then conclude that
    \begin{align*}
        0\leq 
        b - D
        - A(1 +2\alpha)x_1^2 
        - 2(C-A)\alpha x_1 x_2
        - C(1 +2\alpha) x_2^2 \\
        + 2C\alpha x_1^3 x_2
        + 2(A+C)\alpha x_1^2 x_2^2
        - 2A\alpha x_1 x_2^3
    \end{align*}
    for all $\vec{x}=(x_1,x_2)\in\R^2$. Since $C\alpha>0$, however,
    this inequality cannot be true because the polynomial on its right-hand side is cubic in $x_1$ for fixed $x_2\neq 0$, so it is not sign-definite.
\end{proof}
\section{Absorbing ellipsoid for the PDE \texorpdfstring{\cref{e:rd-pde}}{(\ref{e:rd-pde})}}
\label{app:pde-boundedness}
We now show that the PDE~\cref{e:rd-pde} has an absorbing ellipsoid in its state space $\mathcal{X} := L^2(-1,1) \times L^2(-1,1)$.
\begin{proposition}
   Let $v(r,s) = \int_{-1}^1 (r+2)^2 + 100 (s+4)^2 \,d\xi$. There exists $b \in \R$ such that the set $\{(r,s)\in \mathcal{X}: v(r,s)\leq b\}$ is bounded and absorbs solutions of the PDE~\cref{e:rd-pde}.
\end{proposition}
\begin{proof}
    We use an infinite-dimensional analogue of \cref{th:lyap-fun}, in which the state space $\R^n$ is replaced by the function space $\mathcal{X}$ and the term $\ip{f}{\nabla v}$ in the Lyapunov inequality is replaced by the derivative $\mathcal{L}v$ of the Lyapunov function $v$ along PDE solutions. Specifically, it suffices to show that our Lyapunov function $v$ is coercive on $\mathcal{X}$ and that the Lyapunov inequality $b - v(r,s) - \alpha \mathcal{L}v(r,s) \geq 0$ holds for all functions $(r,s)\in\mathcal{X}$ for a suitable choice of constants $b$ and $\alpha>0$.
    
    The coercivity of $v$ in $\mathcal{X}$ is clear from its definition.  For the Lyapunov inequality, we first compute $\mathcal{L}v$ by differentiating the map $t\to v(r(\cdot, t), s(\cdot, t))$ in time along solutions of the PDE \cref{e:rd-pde}. A straightforward computation using integration by parts and the boundary conditions $r(\pm1) = -2$ and $s(\pm1)=-4$ gives
    \begin{align*}
        \mathcal{L}v(r,s) 
        &=
        \int_{-1}^1 2(r+2) \, r_t + 200\, (s+4)  \, s_t \,\de \xi\\
        &=
        \int_{-1}^1 - 2D |r_\xi|^2 - 200(4r +2r^2 + 3r^3 + r^4) - 200 D |s_\xi|^2 + 402 s + 8 \, \de\xi.
    \end{align*}
    We then search for $\alpha>0$ and $b\in\R$ such that the Lyapunov inequality
    \begin{align}\label{e:lyap-pde-explicit}
        0 
        \leq 
        b 
        &- \int_{-1}^1 (r+2)^2 + 100 (s+4)^2 \, \de\xi\\\nonumber
        &+
        2\alpha \int_{-1}^1 D |r_\xi|^2 + 100(4r +2r^2 + 3r^3 + r^4) + 100 D |s_\xi|^2 - 201 s - 4
        \,\de\xi
    \end{align}
    holds for all functions $(r,s)\in\mathcal{X}$. Now, since $r(-1)=-2$ we can use the Cauchy--Schwarz inequality to estimate
    \begin{equation*}
        |r(\xi)+2| 
        = \left|\int_{-1}^\xi r_\xi(z) \,\de z\right| 
        \leq \sqrt{\xi+1}  \left(\int_{-1}^\xi |r_\xi(z)|^2 \,\de z \right)^{\frac{1}{2}} 
        \leq \sqrt{\xi+1} \left( \int_{-1}^1 |r_\xi|^2 \,\de\xi\right)^{\frac{1}{2}},
    \end{equation*}
    so $\int_{-1}^1 (r+2)^2 \,\de\xi \leq 2 \int_{-1}^1 |r_\xi|^2 \,\de\xi$. Similarly, $\int_{-1}^1 (s+4)^2 \,\de\xi \leq 2 \int_{-1}^1 |s_\xi|^2 \,\de\xi$.
    Then, inequality \cref{e:lyap-pde-explicit} holds if we can find $\alpha>0$ and $b\in\R$ such that
    \begin{equation*}
        \int_{-1}^1 
                (\alpha D-1) \left( |r_\xi|^2 + 100|s_\xi|^2 \right) + 100\alpha (4r +2r^2 + 3r^3 + r^4)
               - 201\alpha s - 4\alpha
            \,\de \xi
            \geq -\frac{b}{2}
    \end{equation*}
    for all functions $(r,s)\in \mathcal{X}$. This is possible for any $\alpha > 1/D$ because, for such a choice, the left-hand side of the last inequality attains a finite minimum on the function space $\mathcal{X}$ (this is a standard result in the calculus of variations; see, e.g., \cite{rindler2018}).
\end{proof}

\end{document}